\documentclass[a4paper,12pt]{amsart}
\usepackage[utf8]{inputenc}
\usepackage[T1]{fontenc}
\usepackage{amsmath,amsthm,amssymb,dsfont,a4wide, enumerate,color,mathtools, ulem}
\usepackage[colorlinks]{hyperref}
\usepackage[top=3cm, bottom=3cm, left=3cm, right=3cm]{geometry}
\usepackage{stmaryrd}
\usepackage{bm}

\newtheorem{theorem}{Theorem}[section]
\newtheorem{definition}[theorem]{Definition}
\newtheorem{lemma}[theorem]{Lemma}
\newtheorem{corollary}[theorem]{Corollary}
\newtheorem{hypothese}[theorem]{Hypothesis}
\newtheorem{hypotheses}[theorem]{Hypotheses}
\newtheorem{example}[theorem]{Example}

\newtheorem{remark}[theorem]{Remark}
\newtheorem{proposition}[theorem]{Proposition}

\newtheorem{conjecture}[theorem]{Conjecture}
\def\bthm{\begin{theorem}}
\def\ethm{\end{theorem}}
\def\bcor{\begin{corollary}}
\def\ecor{\end{corollary}}
\def\bprop{\begin{proposition}}
\def\eprop{\end{proposition}}
\def\blem{\begin{lemma}}
\def\elem{\end{lemma}}

\def\brem{\begin{remark}}
\def\erem{\end{remark}}

\def\bdes{\begin{description}}
\def\edes{\end{description}}

\newcommand{\norm}[1]{\left\|#1\right\|}
\newcommand{\abs}[1]{\left|#1\right|}

\newcommand{\related}[1]{\stackrel{#1}{\rightarrow}}

\def\Esp{\mathbb E}
\def\Pr{\mathbb P}

\def\mca{\mathcal A}
\def\mcb{\mathcal B}
\def\mcc{\mathcal C}
\def\mcf{\mathcal F}

\def\mcm{\mathcal M}

\def\mcr{\mathcal R}

\def\mbt{\mathbf t} 

\def\Nset{\mathbb N}
\def\Rset{\mathbb R}

\def\Ind{{\mathds{1}}}
\def\one{{\bf 1}}

\def\Vset{\mathcal X}
\def\OEset{\vec{\mathcal E}}
\def\NOEset{\mathcal E}

\def\supp{\text{Supp}}

\title[VRNBW: an example of path formation]{Vertex reinforced non-backtracking random walks: an example of path formation
}
\author{Line C Le Goff and Olivier Raimond}
\thanks{This research has been conducted as part of the project Labex MME-DII (ANR11-LBX-0023-01)}

\begin{document}

\begin{abstract}
This article studies vertex reinforced random walks that are non-backtracking (denoted VRNBW), i.e. U-turns forbidden.
With this last property and for a strong reinforcement, the emergence of a path may occur with positive probability.
These walks are thus useful to model the path formation phenomenon, observed for example in ant colonies.
This study is carried out in two steps.
First, a large class of reinforced random walks is introduced and results on the asymptotic behavior of these processes are proved.
Second, these results are applied to VRNBWs on complete graphs and for reinforced weights $W(k)=k^\alpha$, with $\alpha\ge 1$.
It is proved that for $\alpha>1$ and $3\le m< \frac{3\alpha -1}{\alpha-1}$, the walk localizes on $m$ vertices with positive probability, each of these $m$ vertices being asymptotically equally visited. Moreover
the localization on $m>\frac{3\alpha -1}{\alpha-1}$ vertices is a.s. impossible.
\end{abstract}

\maketitle

	\section{Introduction}

The contributions of this paper are twofold. First, results concerning the asymptotic behavior of a large class of reinforced random walks (RRW) are proved. Second, we present a strongly reinforced random walk, useful to model the path formation phenomenon.

By formation of a path, we mean that after a certain time, the walk only visits a finite number of vertices, always in the same order.
Such phenomena are observed in ant colonies.
For some species, ants deposit pheromones along their trajectories. 
The pheromone is a chemical substance which attracts the ants of the same colony, and thus reinforces the sites visited by the ants. 
Depending on the succession of these deposits, trails appear between important places such as food sources and nest entries.

RRWs on graphs are natural to model such behavior\,: most visited vertices are more likely to be visited again.
They have already been used to study ant behavior (see~\cite{Deneubourg1990,Vittori2006,Garnier2009}).
But as they are usually defined (see~\cite{Davis1990,Tarres2004,Sellke2008,Limic2007,Benaim2013}), one can obtain a localization phenomenon, i.e. only a finite number of points are visited infinitely often, but no path formation is observed\,: there is no fixed order with which these vertices are visited by the walk.

Therefore, additional rules are necessary for the emergence of a path. In this paper, we choose to add a non-backtracking constraint\,: the walk cannot return immediately to the vertex it comes from.
More precisely, let $G=(\Vset,\NOEset)$ be a locally finite non-oriented graph without loops, with $\Vset$  the set of its {\it vertices} and ${\NOEset\subset \{\{i,j\}:\; i,j\in \Vset,\,i\neq j\}}$ the set of its {\it non-oriented edges}. 
For $\{i,j\}\in\NOEset$, denote ${i\sim j}$, and for $i\in \Vset$, let $N(i):=\{j\in\Vset:\;j\sim i\}$ be the neighborhood of $i$. Let ${X=(X_n)_{n\ge 0}}$ be a non-backtracking random walk on $G$, i.e. for ${n\ge 0}$, ${X_{n+1}\sim X_n}$ and for ${n\ge 1}$, $X_{n+1}\ne X_{n-1}$. 
We suppose that this walk is vertex reinforced\,:
for ${n\ge 0}$ and ${i\in\Vset}$,
\begin{eqnarray*}
\Pr(X_{n+1}=i|X_0,\cdots,X_n) &=& \frac{W(Z_n(i))}{\sum_{j\sim X_n,\,j\ne X_{n-1}} W(Z_n(j))}\Ind_{i\sim X_{n}}\Ind_{i\neq X_{n-1}},
\end{eqnarray*}
where $Z_n(i)$ is the number of times the walk $X$ has visited $i$ up to time $n$ and $W:\Nset\to\Rset^*_+$ is a reinforcement function.  The walk $X$ is called a {\it vertex reinforced non-backtracking random walk} (VRNBW).
Non-backtracking random walks have first been introduced in Section~5.3 of \cite{Madras1993}, and named later non-backtracking random walks in \cite{Ortner2007}.

A path of length $L$ is a set of $L+1$ vertices $P=\{i_0,i_1,\dots,i_L\}$ such that $\{i_{\ell-1},i_{\ell}\}\in \NOEset$ for all $\ell\in\{1,\dots,L\}$.
A cycle of length $L$ is a path $C=\{i_0,\cdots,i_{L}\}$ of length $L$ such that $i_0=i_L$.
The following result shows that for a strong reinforcement, VRNBWs follow a cycle and thus form a path with positive probability.
\begin{proposition}\label{prop0.1} Let $C=\{i_0,\cdots,i_{L}\}$ be a cycle of length $L\ge 3$, such that  for all $i\in C$, $N(i)\cap C$ contains exactly two vertices.
Suppose also that $W$ is a strongly reinforcement function, i.e. $\sum_{k=0}^\infty \frac{1}{W(k)}<\infty$. 
Then, when $X_0=i_0$, the probability that for all $k\ge 0$ and $\ell\in\{1,\dots,L\}$,  $X_{kL+\ell}=i_\ell$ is positive.
\end{proposition}
\begin{proof}
For $i\in \Vset$, let $d_i:=|N(i)|$ be the number of neighbors of $i$.
It is straightforward to check that 
\begin{eqnarray*}
\Pr\big(\forall k\ge 0,\forall \ell \in\{1,\dots,L\},\; X_{kL+\ell }=i_\ell\big)&=& \prod_{k=1}^\infty \prod_{\ell=1}^{L} \left(\frac{W(k)}{W(k)+a_\ell}\right)\;,
\end{eqnarray*}
where $a_\ell=(d_{i_\ell}-2)W(0)$. Since $\sum_{k=0}^\infty \frac{1}{W(k)}<\infty$, this probability is positive.
\end{proof}

The general study of RRWs (in order to obtain almost sure properties) is difficult.
Even without the non-backtracking constraint, almost sure localization on two vertices could only be proved recently by C.~Cotar and D.~Thacker in \cite{Cotar2017} for {\it vertex reinforced random walks} (VRRWs) on connected non-oriented graphs of bounded degree with a reinforcement function $W$ satisfying $\sum_{k=0}^\infty\frac{1}{W(k)}<\infty$. 

Using stochastic algorithm techniques and more precisely results from \cite{Benaim2010}, a more complete study of VRRWs on complete graphs, with reinforcement function $W(k)=(1+k)^\alpha$, with $\alpha\ge 1$, could be done by R.~Pemantle in \cite{Pemantle1992} in the case ${\alpha=1}$, and by M.~Bena\"im, O.~Raimond and B.~Schapira \cite{Benaim2013} in the case ${\alpha >1}$.
The principle of these methods is to prove that the evolution of the empirical occupation measure of the walk is well approximated by an {\it ordinary differential equation} (ODE).
To make that possible some hypotheses are made so that for large times,  the walk behaves almost as an indecomposable Markov chain, whose mixing rate is uniformly bounded from below.

Because of the non-backtracking constraint, this last property fails for VRNBWs. To overcome this difficulty, we set up a framework which is a large particular case of the one in \cite{Benaim2010}. 
More precisely, we introduce a class of RRWs, which contains vertex and edge reinforced random walks (eventually non-backtracking) on non-oriented graphs.
In order to introduce a dependence on the previously visited vertex,  a walk in this class is defined via a process on the set of edges.
This was not necessary in \cite{Benaim2013}.
Moreover at each time step, what is reinforced is a function of the edge that has just been traversed. 
We prove a result similar to Theorem~2.6 of \cite{Benaim2010} (approximation by an ODE), but under different assumptions.

Applying these results, we study VRNBWs on the complete graph with $N\ge 4$ vertices and reinforcement function $W(k)=(1+k)^\alpha$, $\alpha\ge 1$.
Such VRNBWs are then equivalent to urns with $N$ colors, the two last chosen colors being forbidden. 
Note that, for a complete graph, the sets $C$ as in Proposition~\ref{prop0.1} are the sets constituted by three different vertices.

\smallskip
Let us now state our main result for VRNBWs.
Denote by $S\subset \Vset$ the set of infinitely often visited vertices by $X$.The non-backtracking assumption implies that $|S|\ge3$  and that a cycle has been selected only when $|S|=3$.
\begin{theorem}\
\label{th:frrw:vn CV} Let $X$ be a VRNBW on a complete graph with reinforcement function $W(k)=(k+1)^\alpha$. For $n\ge 1$, set $v_n:=\frac{1}{n}\sum_{k=1}^n \delta_{X_k}$. Then, $v_n$ converges a.s.\ towards the uniform probability measure on $S$ and
\begin{enumerate}
\item \label{th:frrw:vn CV:a=1} when $\alpha=1$, $S=\Vset$ a.s.,
\item \label{th:frrw:vn CV:a>1} 
when $\alpha>1$, for $m\ge 3$
\begin{itemize}
	\item $\Pr( |S|=m)>0$, if $3\le m<\frac{3\alpha-1}{\alpha-1}$,
	\item $\Pr(|S|=m)=0$, if $m>\frac{3\alpha-1}{\alpha-1}$.
	\end{itemize}
\end{enumerate}
\end{theorem}
For VRRWs on complete graphs, 
very similar results (replacing $\frac{3\alpha-1}{\alpha-1}$ by $\frac{2\alpha-1}{\alpha-1}$) are obtained in \cite{Benaim2013}. 

We say that a walk $X$ on a graph forms a cycle of length $L$, if after a random time $T$ there is a set $\{i_1,\dots,i_L\}$ such that for all $k\ge 0$ and $1\le \ell\le L$, $X_{T+kL+\ell}=i_\ell$. 
Note that, on complete graphs, VRNBWs can only form cycles of length three.
Moreover, Theorem \ref{th:frrw:vn CV} implies that when $\alpha>3$, a.s. a cycle of length three is formed.
When $\alpha$ is sufficiently close to $1$, there is a positive probability that localization occurs on larger sets, which in this case means no path is formed.

Proposition~\ref{prop0.1} shows that in more general graphs, VRNBWs provide more elaborated paths (of length larger than 3) with positive probability.
M.~Holmes and V.~Kleptsyn have pointed out to us that non backtracking reinforced random walks may localize on more complicated sets, in particular two cycles joined by one path (see \cite{Cotar2017a}).
In fact, it can be seen (using the techniques developed in the present paper) that VRNBWs may localize with positive probability towards such subgraphs, provided the path contains at least two edges (or three vertices), the two cycles satisfy the conditions of Proposition \ref{prop0.1} and there are no edges joining the two cycles.
This, with Theorem~\ref{th:frrw:vn CV}, allows to write the following conjecture.
\begin{conjecture}
Let $X$ be a VRNBW on a connected non-oriented finite graph, with reinforcement function $W(k)=(1+k)^\alpha$.
Suppose that $\alpha>3$, then a.s. $S$ is either a cycle $C$ as in Proposition \ref{prop0.1} or is of the form $C_1\cup P\cup C_2$, where $C_1$ and $C_2$ are two cycles as in Proposition \ref{prop0.1}, joined by a path $P=\{i_0, \cdots, i_L\}$ of length $L\ge 2$, such that $P\cap C_1=\{i_0\}$, $P\cap C_2=\{i_L\}$ and there are no edges joining the two cycles.
\end{conjecture}
To prove such a conjecture is a difficult task. 
Note that the ordered statistics method used in \cite{Cotar2017} is not likely to be used for VRNBWs.

\smallskip
The phase transitions given by Theorem \ref{th:frrw:vn CV} are interesting for the understanding of ant behavior.
Indeed, when $\alpha>3$, a path is formed.
Thus, if ants were acting like a VRNBW and if they can change their sensibility to pheromones by modulating the parameter $\alpha$, they could either make sure that a path will emerge ($\alpha>3$), or could continue to explore a selected area ($\alpha<3$).
Simulation studies of agent-based models provide similar results (see~\cite{Schweitzer1997,Perna2012}).

\smallskip
The paper is organized as follows.  The main notations of the paper are given in Section~\ref{sec:NM:notations}.
In Section~\ref{sec:a class of RRW}, the class of RRWs, introduced in Section~\ref{sec:a class of RRW:definition}, is studied. The main results are stated in Section~\ref{sec:a class of RRW:main result} and their proofs are given in Sections~\ref{sec:a class of RRW:proof} and \ref{sec:RRW:st_unst_equi}.
In Section~\ref{sec:The FVRRW}, the results of Section~\ref{sec:a class of RRW} are applied to VRNBWs on complete graphs, and Theorem~\ref{th:frrw:vn CV} is proved. 
In Sections~\ref{sec:frrw:hyp:def:vn}, \ref{sec:hyp:K}, \ref{sec:piV} and \ref{sec:QV}, we verify that these VRNBWs satisfy the hypotheses of Section~\ref{sec:a class of RRW}.
This is the most delicate part of this paper, where we had to deal with the fact that the transition matrices of the walk may be very slowly mixing.
A Lyapunov function is defined in Section~\ref{sec:lyap}. 
The description of the set of equilibriums, given in Sections~\ref{sec:equilibriums_a=1}, \ref{sec:equilibriums_a>1}, \ref{SIgmakn} and \ref{cupEkn} is also much more complicated compared to the one done for VRRWs in \cite{Benaim2013}. 

\section{Notations}
	\label{sec:NM:notations} 

Let $A$ be a finite set, often identified to $\{1,\dots,N\}$, where ${N=|A|}$. 

For a map $f:A\to\Rset$, we denote $\min(f)=\min\{f(i):i\in A\}$ and $\max(f)=\max\{f(i):i\in A\}$.
Denote by $\one_{A}$ the map on $A$, which is equal to one everywhere. 

A map $\mu:A\to\Rset$ will be viewed as a (signed) measure on $A$ and for $B\subset A$, $\mu(B)=\sum_{i\in B}\mu(i)$. For a measure $\mu$ on $A$ and $f:A\to \Rset$, set $\mu f=\sum_{i\in A} \mu(i) f(i)\;.$
A \textit{probability measure} on $A$ is a measure $\mu$ such that $\mu(A)=1$ and $\mu(i)\ge 0$ for all ${i\in A}$. The {\it support} of $\mu$, denoted ${\supp(\mu)}$, is the set of all ${i\in A}$ such that ${\mu(i)> 0}$. 
The space $\mathcal{M}_A$ of signed measures on $A$ can be viewed as a Euclidean space of dimension $|A|$, with associated Euclidean norm denoted by $\|\cdot\|$. Subsets of $\mathcal{M}_A$ will be equipped with the distance induced by this norm.

We denote by $\Delta_A$ the set of all probability measures on $A$. 
For $m\le N$, we denote by $\Delta_A^m$ the set of probability measures on $A$, whose support is a subset of $A$ containing exactly $m$ points.
For $\Sigma\subset \Delta_A$, let $\Sigma^m$ be defined by
\begin{align}
\label{def:A^k}
\Sigma^m=\Sigma\cap \Delta_A^m\;.
\end{align}
For $i\in A$, let $\delta_i\in \Delta_A$ be defined by $\delta_i(j)=\delta_{i,j}$, where $\delta_{i,j}=1$ if $i=j$ and $\delta_{i,j}=0$ if $i\ne j$. 

Let $A$ and $B$ be two finite sets and let $T:A\times B\to \Rset$. For a measure $\mu$ on $A$, $\mu T$ is the measure on $B$ defined by
\begin{equation}\mu T(b)=\sum_{a\in A} \mu(a)T(a,b),\quad \hbox{ for $b\in B$}\end{equation}
 and for a map $f:B\to\Rset$, $Tf:A\to\Rset$ is the mapping defined by
\begin{equation} \label{def:Vf(e)}
Tf(a)=\sum_{b\in B} T(a,b)f(b),\quad \hbox{ for $a\in A$}.\end{equation}
For $a\in A$, $T(a)$ is the measure on $B$ defined by $T(a)(b)=T(a,b)$, for $b\in B$. This measure will also be denoted $T(a,\cdot)$ or $T_a$. Note that $T(a)f=T_af=Tf(a)$.

For $T:A\times B\to \Rset$ and $U:B\times C\to \Rset$ with $A$, $B$ and $C$ three finite sets, $TU:A\times C\to \Rset$ is defined by
$$TU(a,c)=\sum_{b\in B} T(a,b)U(b,c),\quad \hbox{ for $(a,c)\in A\times C$}.$$

Let $A$ and $B$ be two finite sets. A {\it transition matrix} from $A$ to $B$ is a map $V:A\times B\to [0,1]$ such that $V_a\in \Delta_B$, for all $a\in A$, and we have
\begin{eqnarray}
\label{eq:V1=1} V \one_{B}&=&\one_A\;.
\end{eqnarray}

A {\it Markov matrix} on a finite set $A$ is a transition matrix from $A$ to $A$.
We denote by $\mathcal{M}_A$ the set of all Markov matrices on $A$.
For $i,j\in A$ and $P\in\mathcal{M}_A$, denote ${i\related{P} j}$ when $P(i,j)>0$. The Markov matrix $P$ is said {\it indecomposable} if there is a set ${\mcr\subset A}$ such that for all $i\in A$ and  $j\in\mcr$, there is a path $(i_1,...,i_n)$ in $A$ for which $i\related{P} i_1 \related{P}\cdots\related{P}  i_n\related{P} j$. The set $\mcr$ is called the {\it recurrent class} of $P$.

It is well known that an indecomposable Markov matrix $P$ has an unique {\it invariant probability measure} ${\pi\in\Delta_A}$ characterized by the relation $\pi P = \pi$. 
Moreover, the generator ${-I+P}$ has kernel $\mathbb{R}\one_{A}$ and its restriction to $\{f:A\to\Rset : \pi f = 0\}$ is an isomorphism. 
It then follows that $I-P$ admits a {\it pseudo inverse} $Q$ characterized by
\begin{align}
\label{def:Q}
\left\{
\begin{array}{ll}
Q\one_{A}=0\;,\\
Q(I-P)=(I-P)Q=I-\Pi\;,
\end{array}\right.
\end{align}
where $\Pi\in\mathcal{M}_A$ is defined by $\Pi(i,j)=\pi(j)$, for $i,j \in A$. In other words, $\Pi$ is the orthogonal projection on $\Rset \one_{A}$ for the scalar product $\langle f, g \rangle_\pi={\sum_{i\in A} f(i) g(i) \pi(i)}$. In particular for all $i\in A$ and $f:A\to \Rset$
\begin{eqnarray}
\label{eq:Pif}
\Pi f(i)&=&\sum_{j\in A} \Pi(i,j) f(j)=\sum_{j\in A} \pi(j)f(j) =\pi f\;.
\end{eqnarray}
Note that $Q\in T\mcm_{A}$, where $T\mcm_{A}$ is the set of maps $q:A\times A\to\Rset$, such that ${\sum_{j\in A} q(i,j)=0}$, for all $i\in A$.

Norms, denoted by $\|\cdot\|$,  on the set of functions on $A$ and on $\mcm_A$ are defined by
\begin{eqnarray}
\label{def:norm}
\norm{f}=\max_{i\in A} |f(i)| &\mbox{and}& \norm{P}=\max_{i,j\in A} |P(i,j)|.
\end{eqnarray}
For $r>0$ and $f:A\to\Rset$, we denote by $\mcb(f,r)=\{g:A\to \Rset:\norm{f-g}\le r\}$  the {\it closed ball} of radius $r$ and centered at $f$ for the norm $\|\cdot\|$.

If $Q\in T\mcm_A$ and $V$ is a transition matrix from $A$ to $B$, then for all $a\in A$, $QV(a):B\to \Rset$ is the measure on $B$ defined by $QV(a)f=QVf(a)$ for $f:B\to\Rset$. Note that $QV(a)(b)=QV(a,b)$.

Let $\Gamma$ be a compact subset of the Euclidean space $\Rset^N$. 
The interior of $\Gamma$ is denoted by $\mathring{\Gamma}$ and its boundary by $\partial \Gamma=\Gamma\setminus \mathring{\Gamma}$. The gradient at $v\in\mathring{\Gamma}$ of a differentiable map $H:\Gamma\to\Rset$ is the vector $\nabla H(v):=\big(\partial_1 H(v),\cdots,\partial_{N} H(v)\big)$, where $\partial_i H$ is the partial derivative of $H$ with respect to its $i$-th coordinate. Let $\langle \cdot,\cdot \rangle$ be the standard scalar product on $\Rset^{N}$.

\section{A class of reinforced random walks}
\label{sec:a class of RRW}
	\subsection{Definition}
	\label{sec:a class of RRW:definition}
Let $G=(\Vset,\NOEset)$ be a finite non-oriented graph. To a non-oriented edge ${\{i,j\}\in \NOEset}$ are associated two {\it oriented edges},  ${(i,j)}$ and $(j,i)$. 
Let $\OEset$ be the set of oriented edges.
Set $\mcm=\mcm_{\OEset}$, the set of Markov matrices on $\OEset$. 
Let $R$ be a finite set, called the {\it reinforcement set}, and set $d=|R|$.
Let $V$ be a transition matrix from $\OEset$ to $R$, and $P:\mathring{\Delta}_R\to \mcm$ be a measurable mapping.

We study here discrete time random processes $((X_n,P_n,V_n))$ defined on $(\Omega, \mcf, \Pr)$, 
a probability space equipped with a filtration $(\mcf_n)_{n\ge 0}$. These processes take their values in $\Vset\times \mcm\times\Delta_R$, are adapted to $(\mcf_n)_{n\ge 0}$ and are such that for all $n\ge 1$,
\begin{itemize}
\item $(X_n,P_n,V_n)$ is $\mcf_n$-measurable for each $n\ge 0$,
\item$E_n:=(X_{n-1},X_n)\in\OEset$ and $V_n=V(E_n)$.
\item The conditional law of $E_{n+1}$ with respect to $\mcf_n$ is $P_n(E_n)$, i.e. 
$$\Pr\big(E_{n+1}=(i,j) | \mcf_n\big)=P_n\big(E_n,(i,j)\big), \hbox{ for all $(i,j)\in \OEset$}.$$
\item $P_n=P(v_n)$, where $v_n\in\Delta_R$ is the {\it reinforcement probability measure} at time $n$ defined by
\begin{align}
v_n=\frac{1}{n+d} \left(1+\sum_{k=1}^nV_k\right)\; .
\end{align}
\end{itemize}
Note that $v_n\in \mathring{\Delta}_R$ for all $n$ and that 
these conditions determine the conditional law of $((X_n,P_n,V_n))$ with respect to $\mcf_1$.

When $R=\Vset$ and $V$ is the transition matrix defined by
$$V((i,j),k)=\delta_j(k),\; \hbox{ for $(i,j)\in\OEset$ and $k\in \Vset$},$$
such walks are said vertex reinforced, the mapping $P$ specifying how the walk is reinforced.
In this case, for each $n$, $V_n=\delta_{X_n}$ and $v_n$ is the empirical occupation measure at time $n$ of the vertices by $(X_n)$.

\begin{example} An example of vertex reinforced random walk is given by $P:\mathring{\Delta}_\Vset\to \mcm$ 
defined by 
$$P(v)\big((i,j),(k,\ell)\big) = \frac{W\big(v(\ell)\big)}{\sum_{\ell'\sim j} W\big(v(\ell')\big)} \Ind_{j=k} $$
with $W:(0,1]\to (0,\infty)$ a continuous function.
\end{example}

When $R=\NOEset$ and $V$ is the transition matrix defined by
$$V((i,j),\{k,\ell\})=\delta_{\{i, j\}}(\{k,\ell\}),\;\hbox{ for $(i,j)\in\OEset$ and $\{k,\ell\}\in\NOEset$},$$
such walks are said edge reinforced, the mapping $P$ specifying how the walk is reinforced.
In this case, for each $n$, $V_n=\delta_{\{X_{n-1},X_n\}}$ and $v_n$ is the empirical occupation measure at time $n$ of the non-oriented edges by $(X_n)$.

\begin{example} An example of edge reinforced random walk is given by $P:\mathring{\Delta}_\Vset\to \mcm$ 
defined by 
$$P(v)\big((i,j),(k,\ell)\big) = \frac{W\big(v(\{j,\ell\})\big)}{\sum_{\ell'\sim j} W\big(v(\{j,\ell'\})\big)} \Ind_{j=k}$$
with $W:(0,1]\to (0,\infty)$ a continuous function.
\end{example}

These are rather usual examples, but our setup includes other reinforced processes, by choosing  different transition matrices $V$. For example, one can take $R=\{A:\;A\subset \Vset\}$ and $V((i,j),A)=1$ if $A=N(j)$ and $V((i,j),A)=0$ otherwise, then it is not the actual visited vertex that is reinforced, but all of its neighbors.

\subsection{Main results of Section~\ref{sec:a class of RRW}}
\label{sec:a class of RRW:main result}
A description of the asymptotic of $(v_n)$ with an ODE is given below in Theorem \ref{th:The_theorem} under the following hypotheses.
\begin{hypotheses}
\label{hyp:K}
There is a compact convex subset $\Sigma$ of $\Delta_R$ such that
\begin{enumerate} 
\item \label{hyp:def}For all $n\ge 1$, $v_n\in\mathring{\Sigma}$.
\item \label{hyp:K:Lips} The map $P$ restricted to $\Sigma\cap\mathring{\Delta}_R$ can be extended to a Lipschitz mapping $P:\Sigma\to \mcm$.
\item \label{hyp:K:ind} The matrix $P(v)$ is indecomposable, for all $v\in\mathring{\Sigma}$.
\end{enumerate}
\end{hypotheses}

Denote by $\mcm_{ind}$ the set of indecomposable Markov matrices on $\OEset$. 
\begin{remark}
The present paper is widely inspired by \cite{Benaim2010}. 
Our set-up is different and permits to study a larger class of reinforced walks.
Indeed, the probability measure $V_n$ does not necessarily belong to $\Sigma$, and in \cite{Benaim2010}, the map $v\mapsto P(v)$ would be a continuous mapping from $\Sigma$ to $\mcm_{ind}$. 
This is not the case here, $P(v)$ may not be indecomposable for all $v\in\partial \Sigma$. 
This gives an additional difficulty in the study of the random process $((X_n,P_n,V_n))$. 
\end{remark}

Hypothesis \ref{hyp:K}-\eqref{hyp:K:ind} permits to define the following functions:
\begin{definition}\label{def:piQ}
For $v\in \mathring{\Sigma}$, 
\begin{itemize} 
\item $\pi(v)$ is the invariant probability measure of $P(v)$;
\item $\pi^V(v)$ is the probability measure on $R$ defined by $\pi^V(v):=\pi(v)V$;
\item $Q(v)\in T\mcm_{\OEset}$ is the pseudo-inverse of $I-P(v)$ (see~\eqref{def:Q}).
\end{itemize}\end{definition}
A consequence of Hypothesis \ref{hyp:K}-\eqref{hyp:K:Lips} is that ${\pi:\mathring{\Sigma}\to\Delta_{\OEset}}$ and ${\pi^V:\mathring{\Sigma}\to \Delta_R}$ are locally Lipschitz. For all $n$, set $\pi_n=\pi(v_n)$ and $\pi^V_n=\pi^V(v_n)$.

\begin{example} If the walk is vertex reinforced, then $\pi^V_n\in\Delta_{\Vset}$ and 
$\pi^V_n=\sum_{i\in\Vset} \pi_n(i,\cdot).$
If the walk is edge reinforced, then  $ \pi^V_n\in\Delta_{\NOEset}$ and 
$$\pi_n^V(\{i,j\}) = \pi_n(i,j) + \pi_n(j,i) - \pi_n(i,i)\Ind_{i=j},\; \hbox{ for all $\{i,j\}\in\NOEset$.}$$
\end{example}

The following hypotheses will also be needed.
\begin{hypotheses}
\label{hyp:th}
\begin{enumerate}
\item \label{hyp:th:pi} The map $\pi^V:\mathring{\Sigma}\to\Delta_{R}$ is continuously extendable to $\Sigma$ and this extension is Lipschitz.
\item \label{hyp:th:Q} For all $e\in\OEset$, the map $v\mapsto Q(v) V(e)$ defined on $\mathring{\Sigma}$ is continuously extendable to $\Sigma$.
\end{enumerate}
\end{hypotheses}

For $u\in\{0,1\}$, set ${T_u\Delta_R=\{v:R\to\Rset:\sum_{r\in R} v(r)=u\}}$.
Since $\Sigma$ is convex and compact, for all $v\in T_1\Delta_R$, there is a unique measure $\mu(v)$ in $\Sigma$ such that $\mu(v)$ is the closest measure to $v$ in $\Sigma$. 
This defines $\mu:T_1\Delta_R\to\Sigma$ which is Lipschitz retraction from ${T_1\Delta_R}$ onto $\Sigma$, i.e. $\mu$ is Lipschitz and its restriction to $\Sigma$ is the identity map on $\Sigma$.

Let $F:T_1\Delta_R\to T_0\Delta_R$ be the {\it vector field} defined by
\begin{align}
\label{def:F}
F(v)=-v+\pi^V(\mu(v)).
\end{align}
Then $F$ is Lipschitz (using Hypothesis~\ref{hyp:th}-\eqref{hyp:th:pi}) and induces a global {\it flow} $\Phi: {\Rset\times T_1\Delta_R}\to T_1\Delta_R$, where for all $v_0\in T_1\Delta_R$, ${t\mapsto \Phi_t(v_0):=\Phi(t,v_0)}$ solves the ODE
\begin{align}
\label{eq:eq_diff} \partial_t\Phi_t(v_0)=F(\Phi_t(v_0)),\qquad \Phi(0,v_0)=v_0.
\end{align}

A set $A\subset \Sigma$ is called {\it invariant} if for all $v_0\in A$, $\Phi(\Rset,v_0)\subset A$. A non empty compact set $A$ is an \textit{attracting set} if there exists a neighbourhood $U$ of $A$ and a function $\mbt:(0,\varepsilon_0)\to\Rset^+$ with $\varepsilon_0>0$ such that for all $\varepsilon<\varepsilon_0$ and $t\ge\mbt(\varepsilon)$,
$\Phi_t(U)\subset A^\varepsilon\,,$
where $A^\varepsilon$ stands for the $\varepsilon$-neighbourhood of $A$. An invariant attracting set is called an \textit{attractor}.

A closed invariant set $A$ is called \textit{attractor free} if there does not exist any subset $B$ of $A$, which is an attractor for $\Phi^A$, the flow $\Phi$ restricted to $A$, defined by $\Phi^A_t(v)=\Phi_t(v)$ for all $t\ge 0$ and $v\in A$.

The limit set of $(v_n)$ is the set $L = L((v_n))$ consisting of all points $v = \lim_{k\to\infty} v_{n_k}$  for some sequence $n_k\to\infty$. Note that since $v_n\in\Sigma$ for all $n$, and since $\Sigma$ is compact, then necessarily, $L\subset\Sigma$. The following theorem is similar to Theorem 2.6 of \cite{Benaim2010}.

\begin{theorem}
\label{th:The_theorem}
Assume that Hypotheses \ref{hyp:K} and \ref{hyp:th} are verified, then the limit set of $(v_n)$ is attractor free for $\Phi$, the flow induced by $F$.
\end{theorem}
In most examples of interest, Hypothesis \ref{hyp:K} is easily verified. Hypothesis~\ref{hyp:th} may be difficult to check. It should be noted that these hypotheses do not imply Hypotheses 2.1 and 2.2 of \cite{Benaim2010}. 
There are also situations where one can check Hypothesis \ref{hyp:K}, but cannot hope to verify Hypotheses 2.1 and 2.2  of \cite{Benaim2010} (this is the case for VRNBWs studied in section \ref{sec:The FVRRW}).

When there is a strict Lyapunov function for $\Phi$, the set $L$ can be described more precisely. 
To this purpose we define what an equilibrium and a strict Lyapunov function are.

\begin{definition}
\label{def:equilibriums}
An {\it equilibrium} for $F$ is a point $v_*$ such that $F(v_*)=0$. We denote by ${\Lambda=\{v_*\in\Sigma:v_*=\pi^V(v_*)\}}$ the {\it set of equilibriums} for $F$ in $\Sigma$.
\end{definition}

\begin{definition}
A map  $H:\Sigma\to\Rset$ is a {\it strict Lyapunov function} for $\Phi$, if ${\langle \nabla H(v),F(v)\rangle>0}\;,$ for all $v\in\Sigma\setminus\Lambda$.
\end{definition}

The following theorem is a direct application of Proposition~3.27 of \cite{Benaim2005a}.

\begin{theorem}
\label{th:lyapunov}
If Hypotheses \ref{hyp:K} and \ref{hyp:th} hold, if there exists a strict Lyapunov function $H$ for $\Phi$ and if $H(\Lambda)$ has an empty interior, then $L$ is a connected subset of $\Lambda$ and the restriction of $H$ to $L$ is constant.
\end{theorem}

When $|\Lambda|<\infty$, the connected subsets of $\Lambda$ are singletons and we have the following corollary.

\begin{corollary}
\label{cor:lyapunov}
If Hypotheses  \ref{hyp:K} and \ref{hyp:th} hold, if there exists a strict Lyapunov function $H$ for $\Phi$ and if $\Lambda$ is a finite set, then $v_{\infty}:=\lim_{n\to\infty} v_n$ exists and $v_{\infty}\in\Lambda$.
\end{corollary}

In Section~\ref{sec:RRW:st_unst_equi} we will discuss about the convergence of $v_n$ towards an equilibrium according to its stability. More precisely we will prove under some additional assumptions the convergence of $v_n$ towards any stable equilibrium with positive probability and the non-convergence of $v_n$ towards unstable equilibriums.

	\subsection{Proof of Theorem \ref{th:The_theorem}}
	\label{sec:a class of RRW:proof}
	
Using the fact that 
$(n+d+1)v_{n+1}-(n+d)v_n={V}(E_{n+1})$, 
we write the sequence $(v_n)$ as a stochastic algorithm of step $1/(n+d)$ :
\begin{align}
\label{eq:v(n+1)_v(n)}
v_{n+1}-v_n&=\frac{1}{n+d+1}\big( F(v_n)+U_{n+1} )\;,
\end{align}
with $U_{n+1}= V(E_{n+1})-\pi^V(v_n).$

To prove Theorem \ref{th:The_theorem} we will use Proposition~5.1 in \cite{Benaim2010}. In the following lemma, we restate this proposition in our setting (in Proposition~5.1 of \cite{Benaim2010}, the corresponding notations are $\tau_n:=\sum_{k=0}^n \frac{1}{k+d}\sim \log(n)$, $ m(t):=\sup\{k\ge0:t\ge \tau_k\}\sim e^t$ and $m(\tau_n+T)\sim ne^T$, for $T>0$).
\begin{lemma}
\label{lem:lem_th}
Assume that for all $T\in\Nset^*$,  
\begin{align}
\label{eq:lem_th}
\lim_{n\to+\infty}\; \sup_{n\le k \le n T} \norm{\sum_{q=n}^{k}\frac{U_{q}}{q}}=0\;,
\end{align}
then the limit set of $(v_n)$ is attractor free for the dynamics induced by $F$.
\end{lemma}

\begin{remark} 
Actually Proposition 5.1 of \cite{Benaim2010} states that $L$ is an internally chain transitive set. But a set is internally chain transitive if and only it is attractor free. This result comes from the theory of asymptotic pseudo-trajectories.
For more details, we refer the reader to \cite{Benaim2005a} and precisely to Section~3.3 for the definitions and to Lemma~3.5 and Proposition~3.20 for the equivalence.
\end{remark}

Lemma~\ref{lem:lem_th} implies that Theorem \ref{th:The_theorem} holds as soon as \eqref{eq:lem_th} holds for all $T\in\Nset^*$.

\begin{lemma}
If Hypotheses \ref{hyp:K} and \ref{hyp:th} hold,
then \eqref{eq:lem_th} is verified for all ${T\in\Nset^*}$.
\end{lemma}

\begin{proof}
Along this proof, $C$ is a non-random positive constant that may vary from lines to lines.
For all $n$, set $Q_n=Q(v_n)$ and $\Pi_n=\Pi(v_n)$ and recall that $\pi_n=\pi(v_n)$ and $\pi^V_n={\pi^V}(v_n)$. Remark that for all $e\in\OEset$, $\Pi_n{V}(e)=\pi_n {V}=\pi^V_n$ by using \eqref{eq:Pif}.

Hypotheses~\ref{hyp:th} and the compactness of $\Sigma$ imply that the maps ${v\mapsto{\pi^V}(v)}$ and $v\mapsto Q(v) V(e)$ are uniformly continuous on $\Sigma$, for all $e\in\OEset$. Thus, using that $\|v_{n+1}-v_n\|\le C/n$, we have that,
\begin{align}
&\label{eq:QV}\norm{Q_n{V}(E_n)}\le C\;, \hbox{ for } n\ge 1\\
&\label{eq:pi_piV_QV}
\lim_{n\to\infty}\left\{\norm{(Q_{n+1}-Q_n){V}(E_n)}+\norm{\pi^V_{n+1}-\pi^V_n}\right\}=0\;.
\end{align}
Moreover, for $n$ a positive integer, we have (using the definition of $Q_n$)
\begin{eqnarray}
U_{n+1}=(I-\Pi_n){V}(E_{n+1})
=(Q_n-P_nQ_n){V}(E_{n+1})
= \epsilon_{n+1}+r_{n+1}\;, \label{def:U_n}
\end{eqnarray}
where
\begin{eqnarray}
\label{def:epsilon_n}
\epsilon_{n+1} &=&Q_n{V}(E_{n+1}) -  P_n Q_n{V}(E_n)\;, \\
\label{def:r_n}
r_{n+1}&=& r_{n+1,1}+r_{n+1,2}+r_{n+1,3}\;,
\end{eqnarray}
with
\begin{eqnarray*}
r_{n+1,1} &=&\left(1-\frac{n+1}{n}\right) P_n Q_n{V}(E_n)\;,\\
r_{n+1,2} &=&\frac{n+1}{n}P_n Q_n{V}(E_n)-P_{n+1} Q_{n+1}{V}(E_{n+1})\;,\\
r_{n+1,3} &=&P_{n+1} Q_{n+1}{V}(E_{n+1})-P_nQ_n{V}(E_{n+1})\;.
\end{eqnarray*}

For $T\in\mathbb{N}^*$, $n\in\Nset^*$ and ${1\le i\le 3}$, set
$$ \epsilon_{n}(T)=\sup_{n\le k\le nT} \norm{\sum_{q=n}^{k}\frac{\epsilon_{q}}{q}}\quad \hbox{ and }\quad
r_{n,i}(T)=\sup_{n\le k\le nT} \norm{\sum_{q=n}^{k}\frac{r_{q,i}}{q}}.
$$
Then \eqref{eq:lem_th} is verified as soon as almost surely,  $\lim_{n\to\infty}\epsilon_n(T)=0$  and ${\lim_{n\to\infty}r_{n,i}(T)=0}$ for $i\in\{1,2,3\}$.

The sequence $(\epsilon_{n+1})$ is a martingale difference.
Indeed, for all $n\in\Nset^*$,
\begin{align}
\label{eq:eps_martingale}
\Esp[Q_n{V}(E_{n+1})|\mcf_n]&=P_nQ_n V(E_n)\;.
\end{align}
And using \eqref{eq:QV} we have for all $n\in\Nset^*$,
 $$\norm{\epsilon_{n+1}}\le\|Q_n{V}(E_{n+1})\|+\|P_n\|\,\|Q_n{V}(E_n)\|\le{2C}\;.$$
Moreover, applying Azuma's inequality (\cite{McDiarmid1989}, Theorem 6.7 and \S  6-(c)), we have for all ${\beta>0}$ and all positive integer $n$, 
\begin{align*}
\Pr(\epsilon_n(T)\ge \beta)\le 2|R|  \exp\left(\frac{-\beta^2}{C\sum_{q=n}^{nT}q^{-2}} \right) .
\end{align*}
Since $\sum_{q=n}^{nT} q^{-2} \le nT\times n^{-2}=Tn^{-1}$, we have
$\sum_n \Pr(\epsilon_n(T)\ge \beta) <\infty$, and thus, with Borel-Cantelli Lemma, we conclude that $\lim_{n\rightarrow\infty} \epsilon_{n}(T)=0$ a.s.

For $n\in\Nset^*$,  using \eqref{eq:QV},
$r_{n,1}(T)\le C\sum_{q=n}^{nT}q^{-2} \le \frac{C T}{n}$ and $r_{n,2}(T)\le \frac{2C}{n}$.
Since $P_nQ_n=Q_n-I+\Pi_n$ (see \eqref{def:Q}), for $n\ge 1$,
$$r_{n+1,3}=(Q_{n+1}-Q_n){V}(E_{n+1})+(\pi^V_{n+1}-\pi^V_n)\;,$$
which implies that
$$r_{n,3}(T)\le \log(2T)\left\{\norm{(Q_{n+1}-Q_n){V}(E_n)}+\norm{\pi^V_{n+1}-\pi^V_n}\right\}\;.$$
Therefore, by using \eqref{eq:pi_piV_QV},  this proves that, for $i\in\{1,2,3\}$, $\lim_{n\to\infty}r_{n,i}(T)=0$.
\end{proof}

\subsection{Stable and unstable equilibriums}
	\label{sec:RRW:st_unst_equi}
To define the stability of an equilibrium, we assume the following hypothesis.
\begin{hypothese}
The map ${\pi^V}:\Sigma\to\Delta_R$ is $\mcc^1$.
\end{hypothese}

For $v\in \Sigma$, denote by $DF(v)$ the differential of $F$ at $v$, and, for $u\in T_0\Delta_R$, $D_uF(v):=DF(v)(u) \in T_1\Delta_R$ is the derivative of $F$ at $v$ in the direction $u$. 

\begin{definition}
Let $v_*$ be an equilibrium. We say that $v_*$ is {\rm stable} if all eigenvalues of $DF(v_*)$ have a negative real part and $v_*$ is {\rm unstable} if there exists at least one eigenvalue of $DF(v_*)$ with a positive real part.
\end{definition}

\begin{remark}
\label{rem:st_equi_attractor}
If $v_*$ is a stable equilibrium, then $\{v_*\}$ is an attractor.
\end{remark}

\begin{definition}
Let $v_*$ be an equilibrium. A {\rm stable (unstable) direction} of $v_*$ is an eigenvector of $DF(v_*)$ associated to an eigenvalue with negative (positive) real part.
\end{definition}

\begin{remark}
All eigenvectors of $DF(v_*)$, with $v_*$ a stable equilibrium,  are stable directions and an unstable equilibrium always has at least one unstable direction.
\end{remark}

	\subsubsection{Convergence towards stable equilibriums with positive probability}
	
In this section, it is proved that a stable equilibrium $v_*$ just has to be attainable by $(v_n)$ in order to have that $v_n$ converges towards $v_*$ with positive probability.

\begin{definition}
A point $v_*$ is said attainable by $(v_n)$, if for each ${\epsilon>0}$ and $n_0\in\Nset^*$,
\begin{eqnarray*}
\Pr(\exists n\ge n_0,\norm{v_n-v_*}<\epsilon)&>&0\;.
\end{eqnarray*}
\end{definition}

The following theorem is a particular case of Theorem 7.3 of \cite{Benaim1999} (using Remark~\ref{rem:st_equi_attractor}).

\begin{theorem}\label{th:st_cv}
Let $v_*$ be a stable equilibrium. If $v_*$ is attainable by $(v_n)$, then 
\begin{eqnarray*}
\Pr(v_n\to v_*)&>&0\;.
\end{eqnarray*} 
\end{theorem}

	\subsubsection{Non convergence towards an unstable equilibrium}
\label{sec:non_conv}
Let $v_*\in \Sigma$ be an unstable equilibrium. Then there is an unstable direction $f$ of $v_*$.
Set $P_*=P(v_*)$, $Q_*=Q(v_*)$ and ${\pi_*=\pi(v_*)}$. Set also $\mcr_*=\supp(\pi_*)$, the support of $\pi_*$. For $(i,j)\in\OEset$, let $\mca_{i,j}=\{k\in\Vset:\, P_*\big((i,j),(j,k)\big)>0\}$ 
and ${\mca_j=\bigcup_{i\,:\,(i,j)\in\mcr_*} \mca_{i,j}}$.

\begin{remark}
Let $(E_n^*)=((X_{n-1}^*, X_n^*))$ be a Markov chain of transition matrix $P_*$ and of initial law $\pi_*$. Then $\mca_{i,j}$ is the set of vertices that can be reached by $X^*$ in one step coming from $i$ and starting from $j$, and $\mca_{j}$ is the set of vertices that can be reached by $X^*$ in one step starting from $j$.
\end{remark}

Let $\pi_1$ and $\pi_2$ be the marginals of $\pi_*$, i.e. for all $i,j\in\Vset$,
$$\pi_1(i)=\sum_k \pi_*(i,k)\quad \hbox{ and }\quad \pi_2(j)=\sum_k \pi_*(k,j).$$
Denote by $\mca$ the support of $\pi_1$.
\begin{lemma}
\label{lem:pi1=pi2}
We have $\pi_1=\pi_2$ and $\mcr_*=\{(i,j)\in\OEset:i\in\mca,j\in\mca_i\}$.
\end{lemma}
\begin{proof}
Let $(E_n^*)=((X_{n-1}^*, X_n^*))$ be a Markov chain of transition matrix given by $P_*$ and of initial law $\pi_*$. We know that, for all $n\ge1$, the law of $E_n^*=(X_{n-1}^*,X_{n}^*)$ is $\pi_*$, hence the law of $X_{n-1}^*$ is $\pi_1$ and the law of $X_{n}^*$ is $\pi_2$. Thus $\pi_1=\pi_2$.

Since $\pi_*=\pi_* P_*$, then $(j,k)\in\mcr_*$ if and only if there exists $i\in\Vset$, such that ${\pi_*(i,j)P_*((i,j),(j,k))>0}$. This is equivalent to the fact that there exists ${i\in\Vset}$, such that $(i,j)\in\mcr_*$ and $k\in\mca_{i,j}$, i.e. $k\in\mca_j$.

Note finally that $\mca_j$ is not empty if and only if $j\in\supp(\pi_2)\big(=\supp(\pi_1)\big)$.
\end{proof}

\begin{lemma}
\label{lem:K_*_S_*} There exists $m\ge 0$, such that for all $e\in\OEset$, 
\begin{eqnarray}
\label{eq:K_*_S_*}
\mcr_*&\subset&\bigcup_{q=0}^m \supp\big( P^{q}_*(e,\cdot)\big)\;.
\end{eqnarray}
\end{lemma}
\begin{proof}
Since $\mcr_*$ is the unique recurrent class of $P_*$ and $|\OEset|<\infty$, there exists ${m\ge 1}$ such that for all $e\in\OEset$ and $e'\in\mcr_*$, there exists $q\le m$ for which $P^{q}_*(e,e')>0$.
\end{proof}

\begin{hypotheses}
\label{hyp:unstb_eq}
\begin{enumerate}
\item \label{hyp:unstb_eq:ball} There exists $r>0$, such that ${v\mapsto Q(v){V}(e)}$ is Lipschitz on ${\mcb(v_*,r)\cap \Sigma}$, for all $e\in\OEset$.
\item \label{hyp:unstb_eq:A_y} For all $j\in\mca$ and $k,k'\in\mca_j$, there exists $i\in\Vset$ such that $(i,j)\in\mcr_*$ and ${k,k'\in\mca_{i,j}}$.
\item \label{hyp:unstb_eq:G_g} There doesn't exist any map $g:\mca\to\Rset$, such that $$(i,j)\mapsto{V}f(i,j)-g(i)+g(j) \quad \hbox{is constant on $\mcr_*$}.$$
\end{enumerate}
\end{hypotheses}

This section is devoted to the proof of the following theorem.

\begin{theorem}
\label{th:unst_non_cv} Let $v_*$ be an unstable equilibrium. If Hypotheses \ref{hyp:K}, \ref{hyp:th} and \ref{hyp:unstb_eq} hold, then 
\begin{eqnarray*}
\Pr(v_n\to v_*)&=&0\;.
\end{eqnarray*}
\end{theorem}

\begin{proof}
Along this proof, $C$ will denote a non-random positive constant that may vary from lines to lines.
Equations \eqref{eq:v(n+1)_v(n)} and \eqref{def:U_n} imply that
$$v_{n+1}-v_n=\frac{1}{n+d+1} \big( F(v_n) + \epsilon_{n+1} + r_{n+1}\big)\;.$$ 
The expression of $\epsilon_{n+1}$ and $r_{n+1}$ are given by \eqref{def:epsilon_n} and \eqref{def:r_n} and we recall that ${\Esp[\epsilon_{n+1} |\mcf_n]=0}$ (see \eqref{eq:eps_martingale}). For $n\in\Nset$, set
\begin{align*}
z_n=v_n-\frac{1}{n+d}\big(P_nQ_n V(E_n)\big)\;.
\end{align*}
Note that $z_n\in T_1\Delta_R$. Indeed, using \eqref{eq:V1=1} and the definition of $Q_n$ (see \eqref{def:Q})
\begin{eqnarray*}
P_nQ_n{V}\one_R(E_n)&=&P_nQ_n\one_{\OEset}(E_n)=0\;.
\end{eqnarray*}

The sequence $(z_n)$ is a stochastic algorithm of step $1/(n+d)$\,:  for all $n$, 
\begin{eqnarray*}
z_{n+1}-z_n
&=&\frac{1}{n+d+1} \big( F(z_n) + \epsilon_{n+1} + \tilde r_{n+1}\big)\;,
\end{eqnarray*}
where 
\begin{eqnarray*}
\tilde{r}_{n+1}&=&F(v_n)-F(z_n)+ r_{n+1,1}+r_{n+1,3}\;.
\end{eqnarray*}

By using \eqref{eq:QV}, $\norm{z_n-v_n}\le C/n$ so that  $\{z_n\to v_*\}=\{v_n\to v_*\}$.
Thus, to prove Theorem \ref{th:unst_non_cv}, we will apply Corollary 3.IV.15, p.126 in~\cite{Duflo1996} to  $(z_n)$.

\begin{lemma}
On $\{v_n\to v_*\}$, we have $\tilde{r}_{n+1}=O\big(\frac{1}{n}\big)$.
\end{lemma}

\begin{proof}
Hypothesis~\ref{hyp:th}-\eqref{hyp:th:pi} implies that $F$ is Lipschitz on $T_1\Delta_R$. Thus we have that $\norm{F(v_n)-F(z_n)}\le C/n$, for all ${n\ge 1}$. We also have $\norm{r_{n+1,1}}\le C/n$ (see \eqref{eq:QV}). 

Let $e\in\OEset$ and $n_0$ be an integer such that for all $n\ge n_0$, $v_n\in\mcb(v_*,r)\cap\Sigma$, with ${r>0}$ defined as in Hypothesis~{\ref{hyp:unstb_eq}-\eqref{hyp:unstb_eq:ball}}. Let $n\ge n_0$, then using Hypothesis~\mbox{\ref{hyp:unstb_eq}-\eqref{hyp:unstb_eq:ball}}, the map $v\mapsto Q(v){V}(e)$ is Lipschitz on $\mcb(v_*,r)\cap\Sigma$. Since $|\OEset|<\infty$, the Lipschitz constants of these mappings are uniformly bounded in $e\in\OEset$, and
\begin{align*}
\norm{r_{n+1,3} }&=\norm{(Q(v_{n+1})-Q(v_n)){V}(E_{n+1}) +(\pi^V(v_{n+1})-\pi^V(v_n))}\le C/n\;.
\end{align*}
Hence ${\tilde{r}_{n+1}=O\big(\frac{1}{n}\big)}$ on $\{{v_n\to v_*\}}$.
\end{proof}

The previous lemma directly implies that on $\{v_n\to v_*\}$, $\sum_n \norm{\tilde{r}_{n+1}}^2<\infty$.

Let $m$ be a positive integer such that \eqref{eq:K_*_S_*} is verified. To achieve this proof, according to Corollary~3.IV.15, p126 of \cite{Duflo1996}, it remains to show that on ${\{v_n\to v_*\}}$,
\begin{eqnarray}
\liminf_{n\to\infty} \Esp\left[\left.\sum_{q=0}^{m} (\epsilon_{n+q+1}f)^2 \right|\mcf_n\right]>0\;.
\end{eqnarray}

Let $\mu\in\Delta_{\OEset}$ and $G:\OEset\to\Rset$. Define the variance $\hbox{Var}_\mu(G)$ by
\begin{align}
\label{eq:var=var1}
\hbox{Var}_\mu(G)&=\mu G^2-(\mu G)^2\\
&=\label{eq:var=var2}
\frac{1}{2}\sum_{e,e'\in\OEset}\mu(e)\mu(e') \big(G(e)-G(e')\big)^2\;.
\end{align}

Recall that the conditional law of $E_{n+1}$ with respect to $\mcf_n$ is $P_n(E_n,\cdot)=P(v_n)(E_n,\cdot)$. The conditional mean and variance with respect to $\mcf_{n}$ of $Q_n  V f(E_n)$ are respectively $P_nQ_n Vf(E_n)$ and $\Esp[(\epsilon_{n+1}f)^2 |\mcf_{n}]=\varphi_{v_n}(E_n)$, where
\begin{align*}
\varphi_v(e)=\hbox{Var}_{P(v)(e,\cdot)}(Q(v)  Vf)\;,
\end{align*}
for all $v\in\Sigma$ and $e\in\OEset$. We denote $\varphi_{v_*}$ by $\varphi_*$.

\begin{lemma}
\label{lem:phi_v(e) Lips}
For each $e\in\OEset$, the map $v\mapsto \varphi_v(e)$ is Lipschitz on $\mcb(v_*, r)\cap\Sigma$.
\end{lemma}
\begin{proof}
Indeed, $(\mu,G)\mapsto \hbox{Var}_\mu(G)$ is Lipschitz.
Moreover, for all $e\in\OEset$, ${v\mapsto P(v)(e,\cdot)}$ is Lipschitz on $\Sigma$ and by using Hypothesis \ref{hyp:unstb_eq}-\eqref{hyp:unstb_eq:ball}, $v\mapsto Q(v) V f$ is Lipschitz on $\mcb(v_*,r)\cap\Sigma$. We conclude using the property that the composition of two Lipschitz functions is Lipschitz.
\end{proof}

By using several times Lemma~\ref{lem:phi_v(e) Lips}, let us prove that on $\{v_n\to v_*\}$, we have
\begin{equation}
\label{eq:phi0}
\Esp\left[\left.(\epsilon_{n+q+1}f)^2 \right|\mcf_n\right] 
= P^{q}_*\varphi_*(E_n)+O\left(\frac{1}{n}+\norm{v_n-v_*}\right) \hbox{, for all $q\in\{0,\dots,m\}$}.
\end{equation}

We have for all $q\in\{0,\dots,m\}$,
\begin{eqnarray}
\Esp\left[\left.\big(\epsilon_{n+q+1}f\big)^2 \right|\mcf_n\right] &=&\Esp[\varphi_{v_{n+q}}(E_{n+q})|\mcf_n]\nonumber\\
\label{eq:phi1}&=&\Esp[\varphi_{v_{n}}(E_{n+q})|\mcf_n]+\Esp[\big(\varphi_{v_{n+q}}-\varphi_{v_{n}}\big)(E_{n+q})|\mcf_n]\;.
\end{eqnarray}

Let $r>0$ be defined as in Hypothesis \ref{hyp:unstb_eq}-\eqref{hyp:unstb_eq:ball}. Notice that $\{v_n\to v_*\}\subset \bigcup_{n'} \Omega_{n',r}$, where $\Omega_{n',r}=\{v_n\in\mcb(v_*,r/2)\cap\Sigma, \forall n\ge n'\}$, for all $n'\in\Nset^*$. Let $n_1$ be a positive integer such that $\frac{2m}{n_1+d}\le r/2$. Then for all $n\ge n_1$,
\begin{eqnarray*}
\sup_{0\le q \le m} \norm{v_{n+q}-v_n}\le r/2\;.
\end{eqnarray*}
Indeed, for all $q\in\{0,\dots,m\}$,
\begin{eqnarray*}
v_{n+q}-v_n&=&\left(\frac{n+d}{n+d+q}-1\right)v_n+\frac{1}{n+d+q}\sum_{k=n+1}^{n+q}{V}(E_k)\;.
\end{eqnarray*}
Thus $\sup_{0\le q \le m} \norm{v_{n+q}-v_n}\le \frac{2m}{n+d}\le r/2$.

Fix $n\ge n_1$ and $q\in\{0,\dots,m\}$.
On $\Omega_{n_1,r}$, we have $v_n\in\mcb(v_*,r/2)$ and
\begin{eqnarray*}
{\Esp[\big(\varphi_{v_{n+q}}-\varphi_{v_{n}}\big)(E_{n+q})|\mcf_n]}&=&{\Esp[\big(\varphi_{v_{n+q}}-\varphi_{v_{n}}\big)(E_{n+q})\Ind_{\{v_n\in\mcb(v_*,r/2)\}}|\mcf_n]}\;.
\end{eqnarray*}
Since $v_n\in\mcb(v_*,r/2)$ implies that $v_{n+q}\in\mcb(v_*,r)$, using Lemma~\ref{lem:phi_v(e) Lips}, we have that on $\Omega_{n_1,r}$,
\begin{eqnarray}
\label{eq:phi2}
\abs{\Esp[\big(\varphi_{v_{n+q}}-\varphi_{v_{n}}\big)(E_{n+q})|\mcf_n]}&\le&C\;\Esp[\norm{v_{n+q}-v_n}|\mcf_n]\le C/n\;.
\end{eqnarray}
Using again Lemma~\ref{lem:phi_v(e) Lips}, on $\Omega_{n_1,r}$, 
\begin{eqnarray*}
\Esp[\varphi_{v_{n}}(E_{n+q})|\mcf_n]&=&\Esp[\varphi_{*}(E_{n+q})|\mcf_n]+O(\norm{v_n-v_*})\;.
\end{eqnarray*}
Moreover on $\Omega_{n_1,r}$,
\begin{eqnarray*}
\Esp[\varphi_*(E_{n+q})|\mcf_n]&=&\Esp[\Esp[\varphi_*(E_{n+q})|\mcf_{n+q-1}]|\mcf_n]\\
&=&\Esp[P(v_{n+q-1})\varphi_*(E_{n+q-1})|\mcf_n]\\
&=&\Esp[P_*\varphi_*(E_{n+q-1})|\mcf_n]+O\left(\frac{1}{n}+\norm{v_n-v_*}\right)\;,
\end{eqnarray*}
where the fact that $P(v_{n+q-1})(E_{n+q-1},\cdot)$ is the conditional law with respect to $\mcf_{n+q-1}$ of $E_{n+q}$ is used for the second equality and the facts that 
$$\norm{v_{n+q-1}-v_*}\le C/n+\norm{v_n-v_*}$$
and that $P$ is Lipschitz on $\Sigma$ are used for the third equality. Finally by repeating $q$ times the last computations, we have on $\Omega_{n_1,r}$
\begin{eqnarray}
\label{eq:phi3}
\Esp[\varphi_*(E_{n+q})|\mcf_n]&=&P^q_*\varphi_*(E_n)+O\left(\frac{1}{n}+\norm{v_n-v_*}\right)\;.
\end{eqnarray}

Thus by using \eqref{eq:phi1}, \eqref{eq:phi2} and \eqref{eq:phi3}, we obtain that \eqref{eq:phi0} holds on $\Omega_{n_1,r}$. Thus \eqref{eq:phi0} holds on $\{v_n\to v_*\}\subset \bigcup_{n'} \Omega_{n',r}$, 
which implies that on $\{v_n\to v_*\}$,
\begin{eqnarray*}
\Esp\left[\left.\sum_{q=0}^m (\epsilon_{n+q+1}f)^2\right|\mcf_n \right]&\ge& \inf_{e\in\OEset} \left(\sum_{q=0}^{m} P^{q}_*\varphi_*(e)\right) + O\left(\frac{1}{n}+\norm{v_n-v_*}\right)\;.
\end{eqnarray*}
Thus on $\{v_n\to v_*\}$,
\begin{eqnarray*}
\liminf_{n\to\infty} \Esp\left[\left.\sum_{q=0}^{m} (\epsilon_{n+q+1}f)^2 \right|\mcf_n\right]\ge I_*:= \inf_{e\in\OEset} \left(\sum_{q=0}^{m} P^{q}_*\varphi_*(e)\right)\;.
\end{eqnarray*}

We now prove that $I_*>0$. To this purpose, suppose that $I_*=0$. 
This implies that $\varphi_*(e')=0$, for all $e'\in\mcr_*$. Indeed, if $I_*=0$, there is ${e\in\OEset}$, 
such that $P^{q}_*\varphi_*(e)=0$ for all $q\in\{0,\dots,m\}$. 
Thus $\varphi_*(e')=0$ for all ${e'\in\bigcup_{q=0}^{m} \supp(P_*^q(e,\cdot))}$, 
i.e. there is $q\in\{0,\dots,m\}$ such that $P^q_*(e,e')>0$. 
Therefore, using Lemma \ref{lem:K_*_S_*}, we have $\varphi_*(e')=0$ for all $e'\in\mcr_*$.

Set $G=Q_* V f$. Using \eqref{eq:var=var2}, we have that for each $(i,j)\in\mcr_*$,  $\varphi_*(i,j)=0$ implies that $k\mapsto G(j,k)$ is constant on $\mca_{i,j}$. 
Therefore this with Hypothesis~{\ref{hyp:unstb_eq}-\eqref{hyp:unstb_eq:A_y}} imply that for each $j\in \mca$,
there is a constant $g(j)$ such that $G(j,k)=g(j)$ for all $k\in\mca_j$.

On one hand, using \eqref{def:Q},
\begin{eqnarray*}
(I-P_*)G(i,j)&=& V f(i,j)-\Pi_* V f\;,
\end{eqnarray*}
where $\Pi_*(e,e')=\pi_*(e')$, for all $e,e'\in\OEset$. Remark that $\Pi_* V f=\pi^V(v_*)f$ is a constant.
On the other hand,
\begin{eqnarray*}
(I-P_*)G(i,j)&=&g(i)-\sum_{k\in\mca_{i,j}} P_*((i,j),(j,k)) G(j,k)\\
&=&g(i)-g(j)\;.
\end{eqnarray*}

Hence we have proved that if $I_*=0$, then there exists a map $g:\mca\to\Rset$ such that $ V f(i,j)=\pi^V(v_*)f+g(i)-g(j)$ for all $(i,j)\in\mcr_*$. This is impossible by Hypothesis~${\mbox{\ref{hyp:unstb_eq}-\eqref{hyp:unstb_eq:G_g}}}$. Thus $I_*>0$ and $\Pr(v_n\to v_*)=0$.
\end{proof}

\section{Vertex reinforced non-backtracking random walks}
\label{sec:The FVRRW}
	\subsection{Definitions}
	\label{sec:frrw:definition}

Let $G=(\Vset,\NOEset)$ be the complete graph with $N\ge 4$ vertices. 
Then $\Vset=\{1,\dots,N\}$ and $\NOEset=\{\{i,j\}:i,j\in\Vset,i\neq j\}$.
In this section, the reinforcement set $R$ is the set of vertices $\Vset$ and the walk is vertex reinforced. 
Set $$\Sigma=\{v\in\Delta_\Vset: \max(v)\le 1/3+\min(v)\}\;.$$ 
Note that $\partial \Sigma= \{v\in\Delta_\Vset: \max(v)= 1/3+\min(v)\}\cup\{v\in\Sigma: \exists i\in\Vset,v(i)=0\}$. 
\begin{remark}
\label{rem:sigma3} Measures in $\Sigma^3$ (defined in \eqref{def:A^k}) are uniform on a subset of $\Vset$ containing exactly three points.
The support of any measure in $\mathring{\Sigma}$ contains at least four points, i.e. ${\mathring{\Sigma}\subset \Sigma\setminus \Sigma^3}$.
\end{remark}
Let ${V}:\OEset\times\Vset\to\Rset$ be the transition matrix from $\OEset$ to $\Vset$ defined by
\begin{align}
\label{def:frrw:V}
 V ((i,j),k)&=\delta_j(k)\;, \hbox{ for $(i,j)\in\OEset$ and $k\in \Vset$}.
\end{align}
Set $\alpha\ge 1$ and let $P:\mathring{\Delta}_\Vset\to\mcm_{\OEset}$ be the map defined by
\begin{align}
\label{def:frrw:K}
P(v)\big((i,j),(j',k)\big)=\frac{v(k)^\alpha}{\displaystyle\sum_{k'\in\Vset\setminus\{i,j\} }v(k')^\alpha\;} \;\Ind_{j=j'}\Ind_{i\neq k}\;,
\end{align}
for all $v\in\Sigma$ and $(i,j),(j',k)\in\OEset$. 
Let $(X_n,P_n,V_n)$ be a process associated to $V$ and $P$ as in Section \ref{sec:a class of RRW}. 
Then it is easy to check that $X$ is a VRNBW associated to the reinforcement function $W(k)=(1+k)^\alpha$.
Recall that $V_n=\delta_{X_n}$ and $P_n=P(v_n)$, with $v_n$ the empirical occupation measure of the vertices by $X_n$, defined by
\begin{equation}\label{def:vn}v_n=\frac{1}{n+N}\left(1+\sum_{k=1}^n\delta_{X_k}\right).\end{equation}

In this section, we prove Theorem \ref{th:frrw:vn CV} announced in the introduction.
Theorem~\ref{th:frrw:vn CV} is a consequence of Theorem~\ref{th:The_theorem}, Corollary~\ref{cor:lyapunov}, Theorem~\ref{th:st_cv}, Theorem~\ref{th:unst_non_cv}, Lemma~\ref{lem:localization1} and Lemma~\ref{lem:localization2}.
To apply Theorem~\ref{th:The_theorem} to VRNBWs, we verify Hypotheses~\ref{hyp:K} in Sections~\ref{sec:frrw:hyp:def:vn} and~\ref{sec:hyp:K} and Hypotheses~\ref{hyp:th} in Sections~\ref{sec:piV} and~\ref{sec:QV}.
To apply Corollary~\ref{cor:lyapunov}, we prove in Section~\ref{sec:lyap} that there exists a strict Lyapounov function and, in Sections~\ref{sec:equilibriums_a=1}, \ref{sec:equilibriums_a>1}, \ref{SIgmakn} and \ref{cupEkn} that there is a finite number of equilibriums. 
The stability of the equilibriums is also discussed in these sections.
Finally, applying Theorems~\ref{th:st_cv} and~\ref{th:unst_non_cv}, Theorem \ref{th:frrw:vn CV} is finally proved in Section~\ref{prfthm}.

	\subsection{Hypotheses \ref{hyp:K}-(\ref{hyp:def})}
	\label{sec:frrw:hyp:def:vn}
We verify here that $v_n\in\mathring{\Sigma}$ for $n\ge 0$.

\begin{lemma} \label{lem:v<=1/3} For all $n\ge 0$, we have $\max(v_n)\le \frac{1}{3}\times \frac{n+5}{n+N}$.  \end{lemma}
\begin{proof} Set $i\in \Vset$.
For all $n\ge 0$, we have $|\{0\le \ell\le 2\,:\, X_{n+\ell}=i\}|\le 1$. Thus, if ${Z_n(i)}$ denotes the number of times the walk $X$ visits $i$ before time $n$, then for all ${n\ge 0}$, $Z_{n+3}(i)-Z_{n}(i)\le 1$.
Therefore, for all $n\ge 0$, $Z_{3n}(i)\le n$, $Z_{3n+1}(i)\le n+1$ and $Z_{3n+2}(i)\le n+1$. Thus, for all $n\ge 0$, $\max(Z_n)\le (n+2)/3$. The lemma follows from the fact that $\max(v_n)\le (1+\max(Z_n))/(n+N)$. 
\end{proof}

A first consequence of this lemma is that the only possible limit points $v$ of $(v_n)$ are such that $v(i)\le 1/3$ for all $i$.

\begin{proposition}
\label{lem:frrw:hyp:def:vn} $v_n\in\mathring{\Sigma}$ for all $n$.
\end{proposition}
\begin{proof}
Note that $v_n\in\mathring{\Sigma}$ if and only if $\max(v_n)< 1/3+\min(v_n)$.
Lemma \ref{lem:v<=1/3} and the fact that for all $n\ge 0$, $\min(v_n)\ge 1/(n+N)$, 
imply that $\max(v_n)-\min(v_n)\le \frac{1}{3}\times \frac{n+2}{n+N}$ which is lower than $1/3$ since $N\ge 4$.
\end{proof}

	\subsection{Hypotheses \ref{hyp:K}-\eqref{hyp:K:Lips}-\eqref{hyp:K:ind}}
	\label{sec:hyp:K}
Since the denominator of \eqref{def:frrw:K} doesn't vanish for all $v\in\Sigma$, the map $P$ is $\mcc^1$ on $\Sigma$ and Hypothesis \ref{hyp:K}-\eqref{hyp:K:Lips} is verified. 
Hypothesis \ref{hyp:K}-\eqref{hyp:K:ind} directly follows from the proposition below, after remarking that $\mathring{\Sigma}\subset\Sigma\setminus \Sigma^3$.
\begin{proposition}
\label{prop:K_ind_ex1}
The matrix $P(v)$ is indecomposable for all $v \in \Sigma\setminus\Sigma^3$.
\end{proposition}

\begin{proof}
Let $v\in\Sigma\setminus\Sigma^3$, then by Remark \ref{rem:sigma3}, the support of $v$ contains at least four points. We will prove that the matrix $P(v)$ is indecomposable and that its recurrent class is $\mathcal{S}=\{(i,j)\in\OEset:v_i>0, v_j>0\}$. Recall that $(i,i)\notin \OEset$ and that $e\related{P(v)} e'$ means $P(v)(e,e')>0$.
Let $(i_1,i_2)\in\OEset$ and $(i_3,i_4) \in\mathcal{S}$.

\noindent\textbf{Case 1: ${\bf |\{i_1,i_2,i_3,i_4\}|=4}$.}
Since $i_3\notin\{i_1,i_2\}$, $i_4\notin\{i_2,i_3\}$,
$$(i_1,i_2)\related{P(v)} (i_2,i_3)\related{P(v)} (i_3,i_4)\;.$$
\textbf{Case 2: ${\bf |\{i_1,i_2,i_3,i_4\}|=3}$ and ${\bf i_2\neq i_3}$.}
Since the support of $v$ contains at least four points, there exists $i\in \hbox{Supp}(v)\setminus\{i_1,i_2,i_3,i_4\}$. Thus
$$(i_1,i_2)\related{P(v)} (i_2,i) \related{P(v)}(i,i_3) \related{P(v)} (i_3,i_4)\;.$$
\textbf{Case 3: ${\bf |\{i_1,i_2,i_3,i_4\}|=3}$ and ${\bf i_2= i_3}$.}
In this case, $(i_1,i_2)\related{P(v)}(i_3,i_4)\;.$\\
\textbf{Case 4: ${\bf |\{i_1,i_2,i_3,i_4\}|=2}$.}
In this case $\{i_1,i_2\}=\{i_3,i_4\}$. Since the support of $v$ contains at least four points, there exist $i,j\in\hbox{Supp}(v)\setminus\{i_1,i_2\}$ with $i\neq j$. Thus
$$(i_1,i_2)\related{P(v)}(i_2,i)\related{P(v)}(i,j)\related{P(v)}(j,i_3)\related{P(v)}(i_3,i_4)\;.$$

Consequently $P(v)$ is indecomposable for all $v\in \Sigma\setminus\Sigma^3$.
\end{proof}

\begin{remark}
For $v\in\Sigma^3$, the matrix $P(v)$ is not indecomposable. Indeed, $v$ is uniform on exactly three different points $\{i_1,i_2,i_3\}$. There are two irreducible classes $\mcr_1$ and $\mcr_2$, with ${\mcr_1=\{(i_1,i_2),(i_2,i_3),(i_3,i_1)\}}$ and $\mcr_2=\{(i_2,i_1),(i_1,i_3),(i_3,i_2)\}$, and we have
\begin{align*}(i_1,i_2)\related{P(v)}(i_2,i_3)\related{P(v)}(i_3,i_1)\related{P(v)}(i_1,i_2),\\
(i_2,i_1)\related{P(v)}(i_1,i_3)\related{P(v)}(i_3,i_2)\related{P(v)}(i_2,i_1).
\end{align*}
Thus $\mcr_1$ and $\mcr_2$ define two paths for the Markov chain associated to $P(v)$, i.e. vertices $i_1$, $i_2$ and $i_3$ are visited infinitely often, in the same order.
\end{remark}

	\subsection{The invariant probability measure of  $\mathbf{P(v)}$}
	\label{sec:piV} From now on,
for $v\in\Sigma$ and ${i\in\Vset}$, $v(i)$ will be denoted simply by $v_i$. There should not be any confusion with $v_n\in\Sigma$ defined by \eqref{def:vn}.
For $i\ne j\in\Vset$,  let ${H_{i,j}:\Sigma\to\mathbb{R}_+^*}$, ${H_{i}:\Sigma\to\mathbb{R}_+^*}$ and ${H:\Sigma\to\mathbb{R}_+^*}$ be the maps, which to $v\in\Sigma$ associate
\begin{align}
\label{def:ffrw:Hxy} &H_{i,j}(v)=\sum_{k\notin\{i,j\}} v_k^\alpha\;,\\
\label{def:ffrw:Hx} &H_i(v)=\sum_{j,k;\,i\neq j\neq k\neq i} v_j^\alpha v_k^\alpha = \sum_{j\neq i} v_j^\alpha H_{i,j}(v)\;,\\
\label{def:ffrw:H} &H(v)=\sum_{i,j,k;\, i\neq j\neq k\neq i} v_i^\alpha v_j^\alpha v_k^\alpha=\sum_{i}v_i^\alpha H_i(v)\;.
\end{align}

Recall that for $v\in \Sigma\setminus\Sigma^3$, $\pi(v)$ denotes the invariant probability measure of $P(v)$
and that $\pi^V(v)=\pi(v)V$ belongs to $\Delta_\Vset$. For $(i,j)\in\OEset$ and $k\in\Vset$, we use the notation $\pi_{i,j}(v)$ and $\pi^V_k(v)$ respectively for $\pi(v)(i,j)$ and for $\pi^V(v)(k)$. The expression of these measures is explicitly given in the following proposition.
\begin{proposition}
For all $v\in\Sigma\setminus\Sigma^3$, 
\begin{align}
\label{def:frrw:pi}
&\pi_{i,j}(v) = \frac{v_i^\alpha v_j^\alpha H_{i,j}(v)}{H(v)}, \hbox{ for } (i,j)\in\OEset;\\
\label{def:frrw:hat pi}
&\pi^V_k(v)=\frac{v_k^\alpha H_{k}(v)}{H(v)}, \hbox{ for } k\in\Vset.
\end{align}
Moreover Hypothesis \ref{hyp:th}-\eqref{hyp:th:pi} holds.
\end{proposition}

\begin{proof}
For $v\in\Sigma\setminus\Sigma^3$ and $(i,j)\in\OEset$, set $\mu(i,j) = \frac{v_i^\alpha v_j^\alpha H_{i,j}(v)}{H(v)}$. Then $\mu\in\Delta_{\OEset}$ and is invariant for $P(v)$. Indeed,
\begin{align*}
\mu P(v)(i,j)&=\sum_{(i',j')\in\OEset} \mu(i',j')\,P(v)\big((i',j'),(i,j)\big)\\
&=\sum_{(i',j')\in\OEset} \frac{v_{i'}^\alpha v_{j'}^\alpha v_{j}^\alpha}{H(v)}\Ind_{j'=i}\Ind_{i'\neq j}= \frac{v_i^\alpha v_j^\alpha}{H(v)}\sum_{i'\notin\{i,j\}}v_{i'}^\alpha =\mu(i,j).
\end{align*}
The matrix $P(v)$ being indecomposable, we have $\mu=\pi(v)$.

Recall that $V((i,j),k)=\delta_j(k)$ for $(i,j)\in\OEset$ and ${k\in\Vset}$. Hence for all ${k\in\Vset}$, 
\begin{align*}
{\pi_k^V}(v)&= \sum_{(i,j)\in\OEset} \pi_{i,j}(v)\Ind_{j=k} = \sum_{i\neq k} \pi_{i,k}(v) \\
&= \sum_{i\neq k} \frac{v_i^\alpha v_k^\alpha H_{i,k}(v)}{H(v)}= \frac{v_k^\alpha H_k(v)}{H(v)}\;.
\end{align*}
Since $\alpha\ge 1$ and since $H(v)>0$, for all $v\in\Sigma$, it is straightforward to check that the map $\pi^V$ verifies Hypothesis~\ref{hyp:th}-\eqref{hyp:th:pi}.
\end{proof}

	\subsection{The pseudo-inverse of ${\bf I-P(v)}$}
	\label{sec:QV}

In this section we prove that Hypothesis~{\ref{hyp:th}-\eqref{hyp:th:Q}} holds.
Using Proposition~\ref{prop:K_ind_ex1}, we know that $P(v)$ is indecomposable for all $v\in \Sigma\setminus\Sigma^3$. 
Since $P$ is $C^1$ on $\Sigma$, using the implicit function theorem, 
one can prove (as in Lemma 5.1 in \cite{Benaim1997}) that, for $e\in\OEset$, $v\mapsto Q(v)V(e)$ is $C^1$ on $\Sigma\setminus\Sigma^3$. 
It remains to extend this mapping by continuity to $\Sigma^3$, which is the statement of the following proposition (by taking for all $i\in\Vset$, $g=V(\cdot,i):\OEset\to\Rset$ defined by ${g(e)=V(e,i)}$).

\begin{proposition}
\label{prop:Q(v)g CV}
Let $a:\Vset\to\Rset$ and $g:\OEset\to\Rset$ be the map defined by $g(i,j)=a(j)$ for all $(i,j)\in\OEset$. Then, the map $v\mapsto Q(v)g$ is continuously extendable to $\Sigma^3$.
\end{proposition}

\begin{proof}
Since $\Sigma^3$ is a finite set, it suffices to prove that, for all $v^0\in\Sigma^3$,  the limit of $Q(v)g$ as $v\in\Sigma\setminus\Sigma^3$ goes to $v^0$ exists, and 
by symmetry, to prove this only for $v^0=(1/3,1/3,1/3,0,\cdots,0)$, the uniform probability measure on $\{1,2,3\}$.

By abuse of notation, the transpose $W^T$ of a vector $W$ will be denoted by $W$. We use the following vectorial notations for a function $f:\OEset\to\Rset$
\begin{equation}
\label{eq:notation_vect}
\left\{
\begin{array}{ll}
X_1^f=\big(f(3,1),f(1,2),f(2,3)\big)\;, \qquad &X_2^f=\big(f(2,1),f(3,2),f(1,3)\big)\;,\\
Y_\ell^f=\big(f(1,\ell),f(2,\ell),f(3,\ell)\big)\;,\qquad &Z_\ell^f=\big(f(\ell,1),f(\ell,2),f(\ell,3)\big)\;,\\
T_\ell^f=\big(f(4,\ell),\cdots,f(N,\ell)\big)\;, \qquad &U_\ell^f=\big(f(\ell,4),\cdots,f(\ell,N)\big)\;,
\end{array}\right.
\end{equation}
for $\ell\in\{4,\dots,N\}$ and with the convention $f(\ell,\ell)=0$. 
The vectors $X_1^f$ and $X_2^f$ give $f$ for the edges starting from $\{1,2,3\}$ and ending in $\{1,2,3\}$, 
$Y_\ell^f$ gives $f$ for the edges starting from $\{1,2,3\}$ and ending to $\ell$, 
$Z_\ell^f$ gives $f$ for the edges starting from $\ell$ and ending in $\{1,2,3\}$, 
$T_\ell^f$ gives $f$ for the edges starting from $\{4,\dots,N\}$ and ending to $\ell$ 
and $U_\ell^f$ gives $f$ for the edges starting from $\ell$ and ending in $\{4,\dots,N\}$. 
Note that when ${N=4}$, $T_\ell^f=U_\ell^f=0$. Vectors $X_1^f$, $X_2^f$, $(Y^f_\ell)_ {\ell\ge 4}$, $(Z^f_\ell)_ {\ell\ge 4}$ and $(T^f_\ell)_ {\ell\ge 4}$ are enough to describe $f$, but vectors $(U_\ell^f)_{\ell\ge 4}$ will be useful in the following.

A constant vector $(\lambda,\cdots,\lambda)$ will simply be denoted by $\lambda$. 
For $\ell\ge 4$,  $\delta_\ell$ denotes the vector $(\delta_\ell(4),\cdots,\delta_\ell(N))$, where  $\delta_\ell(\ell)=1$ and $\delta_\ell(m)=0$ if $m\ne \ell$.  Set ${h=(a(1),a(2),a(3))}$. Then, for $\ell\in\{4,\dots,N\}$, we have
\begin{eqnarray}
\label{eq:g_vect}
\left\{
\begin{split}
X^g_1&=X^g_2=Z^g_\ell=h\;,\\
Y_\ell^g&=a(\ell)(1,1,1)=a(\ell)\;,\\
T_\ell^g&=a(\ell)(1-\delta_\ell)\;.
\end{split}\right.
\end{eqnarray}

Set $J=\begin{psmallmatrix} 0&1&0\\0&0&1\\1&0&0 \end{psmallmatrix}$. Then  $J^2=\begin{psmallmatrix} 0&0&1\\1&0&0\\0&1&0 \end{psmallmatrix}$ and $J^3=I$.
Set $L_1=\frac13(J+2J^2)$ and $L_2=\frac13(2J+J^2)$.
For $x\in\Rset^3$, set $\overline{x}=\frac{x_1+x_2+x_3}{3}$.  Proposition \ref{prop:Q(v)g CV} is proved as soon as for all $\ell\in\{4,\dots,N\}$
\begin{eqnarray}
\label{eq:Q(v)g CV}
\left\{
\begin{split}
\lim_{v\to v^0,\,v\in\Sigma\setminus\Sigma^3}X_q^{Q(v)g}&\;=\;-L_qh+\overline{h}\;, \qquad q\in\{1,2\}\\
\lim_{v\to v^0,\,v\in\Sigma\setminus\Sigma^3}Y_\ell^{Q(v)g}&\;=\;-\frac{h}{4}+a(\ell)-\frac{3\overline{h}}{4}\;,\\
\lim_{v\to v^0,\,v\in\Sigma\setminus\Sigma^3}Z_\ell^{Q(v)g}&\;=\;\frac{h-\overline{h}}{2}\;,\\
\lim_{v\to v^0,\,v\in\Sigma\setminus\Sigma^3}T_\ell^{Q(v)g}&\;=\;(a(\ell)-\overline{h})(1-\delta_\ell)\;.
\end{split}\right.
\end{eqnarray}

We now prove \eqref{eq:Q(v)g CV}.
Set $\epsilon_i=1-3v_i$ for $i\in\{1,2,3\}$, $\epsilon_\ell=3v_\ell$ for $\ell\ge4$ and $\epsilon=\sum_{i=1}^3 \epsilon_i\left(=\sum_{\ell\ge 4} \epsilon_\ell\right)$.
Remark that $\epsilon_i=O(\epsilon)$ for all $i\in\Vset$. Indeed, $v_\ell\ge 0$ implies $0\le \epsilon_\ell\le\epsilon$  for $\ell\ge 4$. 
Moreover, as $\epsilon$ goes to $0$, $v_i$ is close to $1/3$ for $i\in\{1,2,3\}$ and $v_\ell$ is close to $0$ for ${\ell\ge 4}$, thus $\max(v)=\frac{1}{3}(1-\min_{i\in\{1,2,3\}} {\epsilon_i})$ and $\min(v)=\frac{1}{3}\min_{\ell\ge 4}\epsilon_\ell$. 
Therefore since $v\in\Sigma$, for all small enough $\epsilon$,
$\min_{i\in\{1,2,3\}}{\epsilon_i}+\min_{\ell\ge 4}{\epsilon_\ell}\ge 0$. Since $0\le\epsilon_\ell\le\epsilon$ for $\ell\ge 4$, this means that $\min_{i\in\{1,2,3\}}\epsilon_i\ge -\epsilon$. Since $\epsilon=\sum_{i=1}^3 \epsilon_i$, we have 
$$\epsilon\ge \max_{i\in\{1,2,3\}} \epsilon_i+2\min_{i\in\{1,2,3\}} \epsilon_i\ge \max_{i\in\{1,2,3\}}\epsilon_i-2\epsilon\;.$$
Thus $3\epsilon\ge\max_{i\in\{1,2,3\}} \epsilon_i$ and $\epsilon_i=O(\epsilon)$ for $i\in\{1,2,3\}$.

To lighten the notation, set $X_1=X_1^{Q(v)g}$, $X_2=X_2^{Q(v)g}$, $Y_\ell=Y_\ell^{Q(v)g}$, $Z_\ell=Z_\ell^{Q(v)g}$, $T_\ell=T_\ell^{Q(v)g}$ and $U_\ell=U_\ell^{Q(v)g}$. Recall that $Q(v)g$ is defined by
\begin{equation}
\label{eq:frrw:def Q(v)g}
\left\{ \begin{array}{ll} (I-P(v))Q(v)g=(I-\Pi(v)) g\;,\\ 
\pi(v) Q(v)g=0\;.\end{array}\right.
\end{equation}
To prove \eqref{eq:Q(v)g CV}, we will give an estimate of $Q(v)g$ as $v$ goes to $v^0$ (or equivalently as ${\epsilon\to 0}$). More precisely we will give estimates of $X_1$, $X_2$, $Y_\ell$, $Z_\ell$ and $T_\ell$ as $\epsilon\to 0$.

For all $i,j,k\in\Vset$, such that $|\{i,j,k\}|=3$, denote $p_{i,j,k}=P(v)((i,j),(j,k))$. Remark that $p_{i,j,k}=p_{j,i,k}$.
When $\{i,j,k\}=\{1,2,3\}$, then $p_{i,j,k}$ only depends on $k$. We denote this probability $p_k$. Since $(1-\epsilon_k)^{-\alpha}=1+O(\epsilon)$ as $\epsilon$ goes to $0$, we have
\begin{eqnarray*}
p_{k}&=&\frac{(1-\epsilon_k)^\alpha}{(1-\epsilon_k)^\alpha+\sum_{\ell\ge 4} \epsilon_\ell^\alpha}=\left(1+\frac{\sum_{\ell\ge 4}\epsilon_\ell^\alpha}{(1-\epsilon_k)^\alpha}\right)^{-1}
\;=\;\left(1+(1+O(\epsilon))\sum_{\ell\ge 4} \epsilon_{\ell}^\alpha\right)^{-1}.
\end{eqnarray*}
This implies the Taylor expansion of $p_k$ as $\epsilon$ goes to $0$\,:
\begin{eqnarray*}
p_{k}&=&1-\sum_{\ell\ge 4} \epsilon_{\ell}^\alpha+O(\epsilon^{\alpha+1}).
\end{eqnarray*}

We also have the following Taylor expansions as $\epsilon$ goes to $0$
\begin{eqnarray*}
p_{i,j,\ell}=p_{j,i,\ell}&=&\frac{\epsilon_{\ell}^\alpha}{(1-\epsilon_k)^\alpha+\sum_{l\ge4} \epsilon_{\ell}^\alpha}\;=\;\epsilon_\ell^\alpha+O(\epsilon^{\alpha+1})\;,\\
p_{i,\ell,j}=p_{\ell,i,j}&=&\frac{(1-\epsilon_j)^\alpha}{\displaystyle (1-\epsilon_j)^\alpha+(1-\epsilon_k)^\alpha +\sum_{\ell'\ge4,\ell'\neq \ell} \epsilon_{\ell'}^\alpha} \;=\; \frac{1}{2}+O(\epsilon)\;,\\
p_{i,\ell,m}=p_{\ell,i,m} &=& \frac{\epsilon_m^\alpha}{\displaystyle (1-\epsilon_j)^\alpha + (1-\epsilon_k)^\alpha + \sum_{\ell'\ge4,\ell'\neq \ell} \epsilon_{\ell'}^\alpha}\;=\;\frac{\epsilon_m^\alpha}{2}+O(\epsilon^{\alpha+1})\;,\\
p_{\ell,m,i}=p_{m,\ell,i} &=& \frac{(1-\epsilon_i)^\alpha}{\displaystyle\sum_{i'=1}^3 (1-\epsilon_{i'})^\alpha + \sum_{\ell'\ge4,\,\ell'\notin\{\ell,m\}} \epsilon_{\ell'}^\alpha}\;=\;\frac{1}{3}+O(\epsilon)\;,\\
p_{\ell,m,n}=p_{m,\ell,n}&=&\frac{\epsilon_n^\alpha}{\displaystyle\sum_{i'=1}^3(1-\epsilon_{i'})^\alpha +\sum_{\ell'\ge4,\,\ell'\notin\{\ell,m\}} \epsilon_{\ell'}^\alpha}\;=\;\frac{\epsilon_n^\alpha}{3}+O(\epsilon^{\alpha+1})\;,
\end{eqnarray*}
for $\{i,j,k\}=\{1,2,3\}$ and $\ell$, $m$, $n\ge 4$, with $|\{\ell,m,n\}|=3$.

Let $L_0=\frac13 (I+J+J^2)$ and let $A_1$ and $A_2$ be the matrices
\begin{align*}
A_1=\begin{psmallmatrix}
0&p_{2}&0\\
0&0&p_{3}\\
p_{1}&0&0
\end{psmallmatrix}\;\mbox{ and }&\;\;\;
A_2=\begin{psmallmatrix}
0&0&p_{3}\\
p_{1}&0&0\\
0&p_{2}&0
\end{psmallmatrix}\;.
\end{align*}
Remark that $L_0x=\overline{x}$ for $x\in\Rset^3$.
The following lemma gives Taylor expansions for ${(I-A_1)^{-1}}$ and for ${(I-A_2)^{-1}}$.
\begin{lemma}
\label{lem:(I-A)^-1}
If $p_{1}p_{2}p_{3}\ne 1$, then $I-A_1$ and $I-A_2$ are invertible. Moreover, for $q\in\{1,2\}$
\begin{eqnarray}
\label{taylor:A1}\Big(\sum_{\ell\ge4}\epsilon_\ell^\alpha\Big) (I-A_q)^{-1} &=& (1+O(\epsilon)) L_0-\Big(\sum_{\ell\ge4}\epsilon_\ell^\alpha\Big) L_q+O(\epsilon^{\alpha+1})\;.
\end{eqnarray} 
\end{lemma}

\begin{proof}
Since the determinants of $I-A_1$ and of $I-A_2$ are both equal to $1-p_{1}p_{2}p_{3}$, $I-A_1$ and $I-A_2$ are both invertible when $p_{1}p_{2}p_{3}\neq 1$. When it is the case, we have
\begin{eqnarray*}
(I-A_1)^{-1} &=& (1-p_{1}p_{2}p_{3})^{-1} 
 \begin{psmallmatrix}1 & p_{2} & p_{2}p_{3}\\
p_{1}p_{3}&1& p_{3}\\
p_{1}&p_{1}p_{2}&1\end{psmallmatrix}\\
&=&(1-p_{1}p_{2}p_{3})^{-1} \left( 3L_0 -3\Big(\sum_{\ell\ge4}\epsilon_\ell^\alpha\Big) L_1+O(\epsilon^{\alpha+1})\right)\;.
\end{eqnarray*}
Since  $p_{1}p_{2}p_{3}  = 1-3\left(\sum_{\ell\ge4}\epsilon_\ell^\alpha\right) + O(\epsilon^{\alpha+1}))$, we have
\begin{eqnarray*}
(1-p_1p_2p_3)^{-1} &=& \frac{1}{3} \Big(\sum_{\ell\ge4}\epsilon_\ell^\alpha\Big)^{-1}  \big[1+O(\epsilon)\big]\;.
\end{eqnarray*}
This implies \eqref{taylor:A1} when $q=1$. We prove \eqref{taylor:A1} when $q=2$ by the same way.
\end{proof}

The following lemma gives the Taylor expansion of $\pi(v)$ as $\epsilon$ goes to $0$.
\begin{lemma}\label{lem:pi=O} For $i\ne j \in \{1,2,3\}$, and $\ell\ne m\in\{4\dots,N\}$,
\begin{eqnarray*}
\pi_{i,j}(v)&=&\frac{1}{6}-\frac{1}{3}\sum_{\ell\ge 4}\epsilon_\ell^\alpha+O(\epsilon^{\alpha+1})\;,\\
\pi_{i,\ell}(v)=\pi_{\ell,i}(v)&=&\frac{\epsilon_\ell^\alpha}{3}+O(\epsilon^{\alpha+1})\;,\\
\pi_{\ell,m}(v)&=&O(\epsilon^{\alpha+1})\;.
\end{eqnarray*}
\end{lemma}

\begin{proof}
Recall that for $i\neq j$, 
$\pi_{i,j}(v)=\dfrac{v_i^\alpha v_j^\alpha \sum_{k\notin\{i,j\}}v_k^\alpha}{H(v)}$,
where
$$H(v)=6\sum_{i<j<k}v_i^\alpha v_j^\alpha v_k^\alpha=\frac{6}{3^{3\alpha}}\left(\prod_{i=1}^3(1-\epsilon_i)^\alpha+O(\epsilon^{\alpha})\right)\;
=\;\frac{6}{3^{3\alpha}}+O(\epsilon).$$
Thus for $i,j, k$ such that $\{i,j,k\}=\{1,2,3\}$
\begin{eqnarray*}
\pi_{i,j}(v)&=&\frac{(1-\epsilon_i)^\alpha(1-\epsilon_j)^\alpha \big((1-\epsilon_k)^\alpha+\sum_{\ell\ge4}\epsilon_\ell^\alpha \big)}{
6 \big(\prod_{i'=1}^3(1-\epsilon_{i'})^\alpha \big) \left(1+3\sum_{\ell\ge4}\epsilon_\ell^\alpha +O(\epsilon^{\alpha+1})\right)}\\
&=&\frac{1}{6}\left(1+\frac{\sum_{\ell\ge 4}\epsilon_\ell^\alpha}{(1-\epsilon_k)^\alpha}\right)
\times\frac{1}{1+3\sum_{\ell\ge 4}\epsilon_\ell^\alpha+O(\epsilon^{\alpha+1})}\\
&=&\frac{1}{6}\frac{1+\sum_{\ell\ge4}\epsilon_\ell^\alpha+O(\epsilon^{\alpha+1})}{1+3\sum_{\ell\ge4}\epsilon_\ell^\alpha +O(\epsilon^{\alpha+1})}\\
&=&\frac{1}{6}-\frac{1}{3}\sum_{\ell\ge4}\epsilon_\ell^\alpha+O(\epsilon^{\alpha+1})\;.
\end{eqnarray*}
We also have for $i,j, k$ such that $\{i,j,k\}=\{1,2,3\}$ and $\ell,m\ge 4$, with $\ell\neq m$
\begin{eqnarray*}
\pi_{i,\ell}(v)=\pi_{\ell,i}(v)&=&\frac{(1-\epsilon_i)^\alpha \epsilon_\ell^\alpha \left((1-\epsilon_j)^\alpha+(1-\epsilon_k)^\alpha+\sum_{\ell'\ge 4,\ell'\neq \ell}\epsilon_{\ell'}^\alpha\right)}{ 6\left(\prod_{i'=1}^3(1-\epsilon_{i'})^\alpha \right) \left(1+3\sum_{\ell'\ge4}\epsilon_{\ell'}^\alpha +O(\epsilon^{\alpha+1})\right)}\\
&=&\frac{\epsilon_\ell^\alpha}{3}+O(\epsilon^{\alpha+1})
\end{eqnarray*}
and
\begin{eqnarray*}
\pi_{\ell,m}(v)&=&\frac{\displaystyle\epsilon_\ell^\alpha\epsilon_m^\alpha \left(\sum_{i'=1}^3 (1-\epsilon_{i'})^\alpha+\sum_{\ell'\ge4,\,\ell'\notin\{\ell,m\}}\epsilon_{\ell'}^\alpha\right)}{6\big(\prod_{i'=1}^3(1-\epsilon_{i'}^\alpha) \big) \left(1+3\sum_{\ell'\ge4}\epsilon_{\ell'}^\alpha +O(\epsilon^{\alpha+1})\right)}\;=\;O(\epsilon^{2\alpha}) \;=\; O(\epsilon^{\alpha+1})\;.
\end{eqnarray*}
\end{proof}
The previous lemma permits to give a Taylor expansion for $\pi(v)g$\,:
\begin{eqnarray*}
\pi(v)g&=&\sum_{(i,j)\in\OEset} \pi_{i,j}(v)g(i,j)\\
&=&\sum_{i,j=1\atop i\neq j}^3 \left(\frac{1}{6}-\frac{1}{3}\sum_{\ell\ge 4}\epsilon_\ell^\alpha\right)a_j+\sum_{i=1}^3\sum_{\ell\ge 4} \frac{\epsilon_\ell^\alpha}{3}\big(a_i+a_\ell\big)+O(\epsilon^{\alpha+1})\\
&=& \left(1+O(\epsilon)\right)\overline{h} + \sum_{\ell\ge4}\epsilon_\ell^\alpha a_\ell + O(\epsilon^{\alpha+1})\;.
\end{eqnarray*}

Let us first prove \eqref{eq:Q(v)g CV} in the case $\overline{h}=0$. Denoting $\langle \epsilon^\alpha,a\rangle= \sum_{\ell\ge4}\epsilon^\alpha_\ell a_\ell$, we have
\begin{eqnarray*}
\pi(v)g&=&\langle\epsilon^\alpha,a\rangle+O(\epsilon^{\alpha+1})\;.
\end{eqnarray*}
Let us now express $P(v)Q(v)g$ in function of $Q(v)g$ and using Notations \eqref{eq:notation_vect} and the equation
$$P(v)Q(v)g(i,j)=\sum_{k\notin\{i,j\}} p_{i,j,k} Q(v)g(j,k)\;, \quad \hbox{for $i\ne j\in\Vset$.} $$
Let $i,j,k$ be such that $\{i,j,k\}=\{1,2,3\}$ and $\ell,m\ge 4$, with $\ell\neq m$. Since
\begin{eqnarray*}
P(v)Q(v)g(i,j)&=&p_{k}Q(v)g(i,j) + \sum_{\ell'\ge4} p_{i,j,\ell'}Q(v)g(j,\ell')\;,
\end{eqnarray*}
using the Taylor expansion of $p_k$ and $\big(p_{i,j,\ell'}\big)_ {\ell'\ge 4}$, we have
\begin{eqnarray*}
X_1^{P(v)Q(v)g}&=&A_1X_1+\sum_{\ell'\ge4} \big(\epsilon_{\ell'}^\alpha+O(\epsilon^{\alpha+1})\big) Y_{\ell'}\;,\\
X_2^{P(v)Q(v)g}&=&A_2X_2+\sum_{\ell'\ge4} \big(\epsilon_{\ell'}^\alpha+O(\epsilon^{\alpha+1})\big) Y_{\ell'}\;.
\end{eqnarray*}
Since
\begin{eqnarray*}
P(v)Q(v)g(i,\ell) &=& p_{i,\ell,j} Q(v) g(\ell,j) + p_{i,\ell,k} Q(v) g(\ell,k) + \sum_{m'\ge4,\, m'\neq \ell} p_{i,\ell,m'} Q(v)g(\ell,m')\;,
\end{eqnarray*}
using the Taylor expansion of $p_{i,\ell,j}$, $p_{i,\ell,k}$ and $\big(p_{i,\ell,m'}\big)_{m'\ge 4,m'\neq \ell}$, we have
\begin{eqnarray*}
Y_\ell^{P(v)Q(v)g} &=& \left(1+O(\epsilon)\right)\left(\frac{J+J^2}{2}\right)Z_\ell + O(\epsilon^\alpha)U_\ell\;.
\end{eqnarray*}
Since 
\begin{eqnarray*}
P(v)Q(v)g(\ell,i) &=& p_{\ell,i,j}Q(v)g(i,j) + p_{\ell,i,k}Q(v)g(i,k) + \sum_{m'\ge4,\, m'\neq \ell} p_{\ell,i,m'}Q(v)g(i,m')\;,
\end{eqnarray*}
using the Taylor expansion of $p_{\ell,i,j}$, $p_{\ell,i,k}$ and $\big(p_{\ell,i,m'}\big)_{m'\ge 4,m'\neq \ell}$, we have
\begin{eqnarray*}
Z_\ell^{P(v)Q(v)g} &=& \big(1+O(\epsilon)) \left(\frac{JX_1+J^2X_2}{2}\right) + O(\epsilon^\alpha)\sum_{m'\ge4,\, m'\neq \ell} Y_{m'}\;.
\end{eqnarray*}
Since
\begin{eqnarray*}
P(v)Q(v)g(m,\ell)&=&p_{m,\ell,i}Q(v)g(\ell,i)+p_{m,\ell,j}Q(v)g(\ell,j) \\
& &+\;p_{m,\ell,k}Q(v)g(\ell,k)+\sum_{n\ge4,\,n\notin\{\ell,m\}}p_{m,\ell,n}Q(v)g(\ell,n)\;,
\end{eqnarray*}
using the Taylor expansion of $p_{m,\ell,n}$ for $n\notin\{\ell,m\}$, we have
\begin{eqnarray*}
T_\ell^{P(v)Q(v)g} &=& \big(1+O(\epsilon)\big)\overline{Z_\ell} + O(\epsilon^\alpha)U_\ell\;.
\end{eqnarray*}

Using the previous expansions and the expression of $g$ given by \eqref{eq:g_vect}, the system \eqref{eq:frrw:def Q(v)g} implies that
\begin{align*}
&(I-A_q)X_q \;=\; h-\langle\epsilon^\alpha,a\rangle + \sum_{\ell'\ge4} \epsilon_\ell^\alpha Y_{\ell'} + O(\epsilon^{\alpha+1})\;, \qquad q\in\{1,2\},\\
&Y_\ell \;=\; a_\ell-\langle\epsilon^\alpha,a\rangle +(1+O(\epsilon))\left(\frac{J+J^2}{2}\right)Z_\ell +O(\epsilon^\alpha)U_\ell+O(\epsilon^{\alpha+1})\;,\\
&Z_\ell \;=\; h-\langle\epsilon^\alpha,a\rangle +(1+O(\epsilon))\left(\frac{JX_1+J^2X_2}{2}\right) +O(\epsilon^\alpha)\sum_{m\ge4,m\neq \ell} Y_m+O(\epsilon^{\alpha+1})\;,\\
&T_\ell \;=\; a_\ell\one_\ell-\langle\epsilon^\alpha,a\rangle +(1+O(\epsilon))\overline{Z_\ell} +O(\epsilon^\alpha) U_\ell +O(\epsilon^{\alpha+1})\;,
\end{align*}
for $\ell\ge4$.
Recall the definition of $\norm{\cdot}$ given by \eqref{def:norm} and set $\norm{X}=\sup\{\norm{X_1},\norm{X_2}\}$.
We then have for all $\ell\ge 4$,
\begin{eqnarray*}
Z_\ell&=&h+\frac{JX_1+J^2X_2}{2} +O(\epsilon(1+\norm{X}+\norm{Y}))\;.
\end{eqnarray*}
Remarking that $J\one=J^2\one=\one$, we have
\begin{eqnarray}
\label{eq:Tl=O}
T_\ell&=&a_\ell(1-\delta_\ell)+\frac{\overline{X_1}+\overline{X_2}}{2}+O(\epsilon(1+\norm{X}+\norm{Y}+\norm{U}))\;.
\end{eqnarray}
Remarking that $(J+J^2)h=-h$, we also have
\begin{eqnarray*}
Y_\ell&=&-\frac{h}{2}+a_\ell+\left(\frac{I+J^2}{4}\right)X_1+\left(\frac{I+J}{4}\right)X_2+O(\epsilon(1+\norm{X}+\norm{Y}+\norm{U}))\;.
\end{eqnarray*}
Using Lemma \ref{lem:(I-A)^-1}, remarking that $L_0J=L_0J^2=L_0$, that $L_0b=\overline{b}$ for $b\in\Rset^3$ and recalling that $\overline{h}=0$, we have
\begin{eqnarray}
\nonumber X_1&=&(I-A_1)^{-1} \left[h+\left(\sum_{\ell\ge4}\epsilon_\ell^\alpha\right) \left(\frac{I+J^2}{4}X_1 +\frac{I+J}{4}X_2-\frac{h}{2}\right)\right.\\
\nonumber & &\left.\phantom{\sum_{\ell\ge4}} \qquad\qquad+O\big(\epsilon(1+\norm{X}+\norm{Y}+\norm{U})\big)\right]\\
\label{eq:X1=O}&=&-L_1h+\frac{\overline{X_1}+\overline{X_2}}{2}+O\big(\epsilon(1+\norm{X}+\norm{Y}+\norm{U})\big)\;.
\end{eqnarray}
and likewise
\begin{eqnarray}
\label{eq:X2=O}
X_2&=&-L_2h+\frac{\overline{X_1}+\overline{X_2}}{2}+O\big(\epsilon(1+\norm{X}+\norm{Y}+\norm{U})\big)\;.
\end{eqnarray}
Remarking that $JL_1+J^2L_2=I+L_0$, \eqref{eq:X1=O} and \eqref{eq:X2=O} implies
\begin{eqnarray}
\label{eq:Zl=O}
Z_\ell&=&\frac{h}{2}+\frac{\overline{X_1}+\overline{X_2}}{2}+O(\epsilon(1+\norm{X}+\norm{Y}+\norm{U}))\;.
\end{eqnarray}
Remarking that
\begin{eqnarray*}
(I+J^2)L_1+(I+J)L_2&=&\frac{5}{2}L_0-\frac{I}{2}
\end{eqnarray*}
and using $L_0h=\overline{h}=0$, \eqref{eq:X1=O} and \eqref{eq:X2=O} implies
\begin{eqnarray}
\label{eq:Yl=O}
Y_\ell&=&-\frac{h}{4}+a_\ell+\frac{\overline{X_1}+\overline{X_2}}{2}+O(\epsilon(1+\norm{X}+\norm{Y}+\norm{U}))\;.
\end{eqnarray}
An immediate consequence of \eqref{eq:Tl=O}, \eqref{eq:Zl=O} and \eqref{eq:Yl=O} is that
\begin{eqnarray*}
\norm{Y}&=&O\big(1+\norm{X}\big)\;,\\
\norm{Z}&=&O\big(1+\norm{X}\big)\;,\\
\norm{T}&=&\norm{U}=O\big(1+\norm{X}\big)\;.
\end{eqnarray*}
Since $\pi(v)f=0$, Lemma \ref{lem:pi=O} implies that
\begin{eqnarray*}
\frac{\overline{X_1}+\overline{X_2}}{2}\left(1-2\sum_{\ell'\ge4}\epsilon_{\ell'}^\alpha\right) +\sum_{\ell'\ge4}\epsilon_{\ell'}^\alpha \big(\overline{Y_{\ell'}}+\overline{Z_{\ell'}}\big) +O\big(\epsilon^{\alpha+1}(1+\norm{X})\big)\;.
\end{eqnarray*}
Thus
\begin{eqnarray}
\label{eq:barX1+barX2=O}
\frac{\overline{X_1}+\overline{X_2}}{2}&=&O\big(\epsilon^\alpha(1+\norm{X})\big)\;.
\end{eqnarray}
Using \eqref{eq:X1=O},\eqref{eq:X2=O} and \eqref{eq:barX1+barX2=O}, we get $\norm{X}=O(1)$ and thus that for $ q\in\{1,2\}$ and $\ell\ge 4$,
\begin{eqnarray*}
X_q&=&-L_qh+O(\epsilon)\;, \\
Y_\ell&=&-\frac{h}{4}+a(\ell)+O(\epsilon)\;,\\
Z_\ell&=&\frac{h}{2}+O(\epsilon)\;,\\
T_\ell&=&a(\ell)(1-\delta_\ell)+O(\epsilon)\;.
\end{eqnarray*}

Suppose now that $\overline{h}\neq0$. Set $g_0=g-\overline{h}$. Then since $Q(v)\one=0$, $Q(v)g=Q(v)g_0$ and $\lim_{v\to v^0}Q(v)g=\lim_{v\to v^0}Q(v)g_0$. Note that
\begin{eqnarray}
\left\{
\begin{split}
X^{g_0}_1&=X^{g_0}_2=Z^{g_0}_\ell=h-\overline{h}\;,\\
Y_\ell^{g_0}&=a(\ell)-\overline{h}\;,\\
T_\ell^{g_0}&=(a(\ell)-\overline{h})(1-\delta_\ell)\;,
\end{split}\right.
\end{eqnarray}
for ${\ell\in\{4,\dots,N\}}$. Thus \eqref{eq:Q(v)g CV} holds.
\end{proof}

Since Hypotheses \ref{hyp:K} and \ref{hyp:th} hold, the vector field $F:T_1\Delta_{\Vset}\to T_0\Delta_{\Vset}$, defined by \eqref{def:F} induces a flow $\Phi$ for the differential equation $\dot{v}=F(v)$. Moreover Theorem~\ref{th:The_theorem} holds and the limit set of $(v_n)$ is attractor free for $\Phi$.

	\subsection{A strict Lyapunov function}
	\label{sec:lyap}

\begin{proposition}
\label{prop:lyapunov}
The map $H:\Sigma \to \mathbb{R}_+^*$, defined by \eqref{def:ffrw:H} is a strict Lyapunov function for $\Phi$.
\end{proposition}

\begin{proof}
The map $H$ is $\mcc^1$ on $\Sigma$. For $v\in \Sigma$, set $h(v):\Vset\to \Rset$ such that for $i\in\Vset$,
$h_i(v)=v_i^{\alpha-1}H_i(v)$. Then for $v\in\Sigma$,
$H(v)=\sum_i v_i^\alpha H_i(v)=\sum_i v_i h_i(v)=v h(v)$.

For $i,j\in\Vset$, with $i\neq j$, the maps $H_{i,j}$, $H_i$ and $H$ are defined on $\Sigma$. But we will consider here that they are respectively defined on $\Rset^N$ by \eqref{def:ffrw:Hxy}, \eqref{def:ffrw:Hx} and~\eqref{def:ffrw:H}. 
For $i,j\in\Vset$, we have
\begin{eqnarray}\partial_i H_j(v)=\left\{\begin{array}{ll} 0\;, & \hbox{ if $i=j$}\;,\\
2\alpha v_i^{\alpha-1}H_{i,j}(v)\;, & \hbox{ if $i\ne j$} \end{array}\right.\label{eq:frrw:dHj}\end{eqnarray}
and
\begin{eqnarray}
\partial_i H(v)&=&\sum_{j\neq i}v_j^\alpha \partial_i H_j(v)+\alpha v_i^{\alpha-1}H_i(v)\;,\nonumber\\
\label{eq:frrw:dH}&=&3\alpha v_i^{\alpha-1}H_i(v)\;.
\end{eqnarray}

Thus (using \eqref{def:frrw:hat pi})
\begin{align*}
\langle\nabla H(v), \pi^V(v)-v\rangle
&=3\alpha\left(\sum_i\frac{v_i^{2\alpha-1	}(H_i(v))^2}{H(v)}-\sum_i v_i^{\alpha}H_i(v) \right)\;,\\
&=\frac{3\alpha}{H(v)}\big[ v h^2(v)-\big(v h(v)\big)^2\big]\;,
\end{align*}
which is positive for all $v\in \Sigma\setminus\Lambda$.
This proves that $H$ is a strict Lyapunov function for $\Phi$.
\end{proof}

Hypotheses \ref{hyp:K} and \ref{hyp:th} hold and there is a strict Lyapunov function for $\Phi$. Thus by applying Theorem \ref{th:lyapunov} and Corollary \ref{cor:lyapunov}, if $H(\Lambda)$ has an empty interior, the limit set of $(v_n)$ is a connected subset of $\Lambda$ and if $\Lambda$ is a finite set, then $v_\infty:=\lim_{n\to\infty} v_n$ exists and ${v_\infty\in\Lambda}$.

\subsection{Equilibriums of $F$ when $\alpha=1$ and their stability.}
\label{sec:equilibriums_a=1}
Recall that $F$ is defined on $\Sigma$ and that $DF(v):T_0\Delta_\Vset\to T_0\Delta_\Vset$ is linear.
To calculate $D_uF(v)$, for $u\in T_0\Delta_\Vset$, it will be convenient to view $F$ as a map defined on $\Rset^N$ by \eqref{def:F}.  Note finally that $e_i-e_j\in T_0\Delta_\Vset$ and $e_i-v\in T_0\Delta_\Vset$, for all $i,j\in\Vset$.

\begin{proposition}
\label{prop:frrw:eq,a=1} 
When $\alpha=1$, the equilibriums of $F$ are the uniform probability measures on subsets of $\Vset$ containing at least three vertices. Moreover, the only stable equilibrium is the uniform probability measure on $\Vset$, and any other equilibrium is unstable.
\end{proposition}
\begin{proof}
Note that $v$ is an equilibrium if and only if for all $i\in\supp(v)$, $H_i(v)=H(v)$. 
Let $v\in\Sigma$ be uniform on $A\subset \Vset$. Set $m=|A|\ge 3$. Then it is straightforward to check that, for $i\in A$, $H_i(v)=(m-1)(m-2)m^{-2}=H(v)$. Thus $v$ is an equilibrium.

Let $v$ be an equilibrium.
Then for all $i\in\supp(v)$ and $j\in\supp(v)$, $H_i(v)=H_j(v)$.
For $i\neq j$,
$$H_i(v)-H_j(v)=2(v_j-v_i)\sum_{k\notin\{i,j\}} v_k.$$
Thus, since $|\supp(v)|\ge 3$, $H_i(v)=H_j(v)$ implies that $v_i=v_j$.
This proves that $v$ is uniform.

For $v\in \Sigma$, $F_i(v)=-v_i+\frac{v_i H_i(v)}{H(v)}$ for $i\in \Vset$ and for $i,j\in\Vset$,
\begin{eqnarray*}
\partial_j H_i(v)&=& \left\{\begin{array}{ll} 0\;, &\hbox{when $i=j$}\;,\\
2 H_{i,j}(v)\;, &\hbox{when $i\ne j$}\;,\end{array}\right.\\
\partial_i H(v)&=& 3 H_i(v)\;,  \hbox{ for $i\in \Vset$}.
\end{eqnarray*}
Therefore, for $i,j\in\Vset$,
\begin{eqnarray*}
\partial_i F_j(v) &=& 
\left\{\begin{array}{ll} -1+ \frac{H_i(v)}{H(v)}-3 v_i \left(\frac{H_i(v)}{H(v)}\right)^2\;,  &\hbox{when $i=j$}\;,\\
v_j\left(2 \frac{H_{i,j}(v)}{H(v)}-3 \frac{H_i(v)}{H(v)}\frac{H_j(v)}{H(v)}\right)\;,&\hbox{when $i\ne j$}.\end{array}\right.
\end{eqnarray*}

When $v$ is an equilibrium, for ${i\in\supp(v)}$
\begin{eqnarray*}
\partial_i F_j(v) &=& \left\{\begin{array}{ll} 0\;, &\hbox{if $j\notin\supp(v)$}\;,\\
-3v_i\;, & \hbox{if $j=i$} \;,\\
v_j\left(2\frac{H_{i,j}(v)}{H(v)}-3\right)\;,& \hbox{if $j\in\supp(v)$ with $j\ne i$}.\end{array}\right.
\end{eqnarray*}
and for $i\notin\supp(v)$,
\begin{eqnarray*}
\partial_i F_j(v) &=&\left\{\begin{array}{ll}
0\;,& \hbox{if $j\notin\supp(v)$ with $j\ne i$}\;,\\
 -1+\frac{H_i(v)}{H(v)}\;, & \hbox{if $j=i$}\;,\\
\frac{v_j}{H(v)}\left(2 H_{i,j} (v)-3 H_i(v)\right)\;,& \hbox{if $j\in\supp(v)$}.
\end{array}\right.
\end{eqnarray*}

Since $v$ is uniform on its support, denoting ${m=|\supp(v)|}$, $v_i=1/m$ for all $i\in\supp(v)$ and
\begin{eqnarray*}
H_{i,j}(v)&=&\left\{\begin{array}{ll}(m-2)m^{-1}\;, &\hbox{ for $i,j\in\supp(v)$ with $i\ne j$},\\
(m-1)m^{-1}\;, &\hbox{ for $i\in\supp(v)$ and $j\notin \supp(v)$}\;,\\
1 \;, &\hbox{ for $i,j\notin\supp(v)$ with $i\ne j$},\end{array}\right. \\
H_i(v)&=&\left\{\begin{array}{ll} (m-1)(m-2)m^{-2}\;, & \hbox{ for $i\in\supp(v)$},\\
(m-1)m^{-1}\;, & \hbox{ for $i\notin\supp(v)$}, \end{array}\right.\\
H(v)&=& (m-1)(m-2)m^{-2}\;.
\end{eqnarray*}

Thus for $i\in\supp(v)$,
\begin{eqnarray*}
\partial_iF(v)
&=& -3v_i e_i +\sum_{j\neq i}v_j\left(2\frac{H_{i,j}(v)}{H(v)}-3\right)e_j\;,\\
&=& -3v_i e_i -\frac{m-3}{m-1}\sum_{j\neq i}v_je_j\;,\\
&=& -\frac{m-3}{m-1}v-\frac{2}{m-1} e_i
\end{eqnarray*}
and for $i\notin\supp(v)$,
\begin{eqnarray*}
\partial_iF(v)
&=& \left(-1+\frac{H_i(v)}{H(v)}\right) e_i+\sum_{j\in\supp(v)} \frac{v_j}{H(v)}\big(2 H_{i,j} (v)-3 H_i(v)\big)e_j\;,\\
&=&\frac{2}{m-2} e_i-\frac{m}{m-2}v\;.
\end{eqnarray*}
Therefore for $i,j\in\supp(v)$, we have
\begin{eqnarray*}
D_{e_i-e_j}F(v) &=&- \frac{2}{m-1}(e_i-e_j)
\end{eqnarray*}
and for $i\notin\supp(v)$,
\begin{eqnarray*}
D_{e_i-v}F(v) &=& \frac{2}{m-2}(e_i-v)\;.
\end{eqnarray*}

Hence the spectrum of $DF(v):T_0\Delta_\Vset\to T_0\Delta_\Vset$ is completely described : $-2/(m-1)$ is an eigenvalue of multiplicity $m-1$ and $2/(m-2)$ is an eigenvalue of multiplicity ${N-m}$.
When $m=N$, i.e. when $v$ is uniform on $\Vset$, $-2/(m-1)<0$ is the only eigenvalue of $DF(v)$ and $v$ is stable. Whereas when $m< N$, $2/(m-2)>0$ is an eigenvalue and $v$ is unstable.
\end{proof}

\subsection{Equilibriums of $F$ when $\alpha>1$.}
\label{sec:equilibriums_a>1}
As in the previous section, $F$ is viewed as a map defined on $\Rset^N$ by \eqref{def:F}. 
Note that $v\in\Sigma$ is an equilibrium if and only if for all $i\in\supp(v)$, $v_i^{\alpha-1}H_i(v)=H(v)$.

\begin{proposition} \label{prop:frrw:unif_eq} Uniform probability measures on subsets of $\Vset$ containing at least three vertices are equilibriums for $F$.
\end{proposition} 

\begin{proof} Let $v$ be a uniform probability measure on a subset $A$ of $\Vset$, with $m:=|A|\ge 3$. 
Then, for $i\in A$, $v_i=m^{-1}$ and $v_i^{\alpha-1}H_i(v)=m(m-1)(m-2) m^{-3 \alpha}=H(v)$ and 
$v$ is an equilibrium.
\end{proof}

When $\alpha>1$, the set $\Lambda$ is not explicitly described. However, it is possible to state some properties. To this aim, we introduce some notations.
\begin{definition}\ \label{def:lambda}
\begin{itemize}
\item For $1\le m\le N$, denote by $\mu_m$ the uniform probability measure on $\{1,\dots,m\}$.
\item For $1\le k<m\le N$, let $\Sigma_{k,m}$ be the set of all $p\mu_k+(1-p)\mu_m \in \Sigma$, such that $p\in (0,1)$ and $p/k+(1-p)/m\le 1/3$.
\item For $1\le k<\ell<m\le N$, let $\Sigma_{k,\ell,m}$ be the set of all $p_1\mu_{k}+p_2\mu_{\ell}+p_3\mu_{m}\in\Sigma$, such that $p_i>0$, $\sum_i p_i=1$ and $p_1/k+p_2/\ell+p_3/m\le 1/3$.
\item Let $\vec{\Sigma}$ be the set of all $v\in\Sigma$ such that $1/3\ge v_1\ge v_2\ge\cdots\ge v_N$.
\end{itemize}
\end{definition}
Note that $\Sigma_{k, m}\subset \vec\Sigma$ and $\Sigma_{k, \ell, m}\subset \vec\Sigma$, for all $1\le k<\ell<m\le N$.

\begin{remark}
If $v\in \Lambda$ is an equilibrium and $\sigma$ is a permutation of $\{1,\dots,N\}$, then $v\circ\sigma:=(v_{\sigma(1)}, \cdots, v_{\sigma(N)})$ is also an equilibrium.
Hence, if $v\in\Lambda$, then there is a permutation $\sigma$ such that $v\circ\sigma\in\Lambda\cap\vec{\Sigma}$.
\end{remark}

\begin{proposition} If $v\in\vec{\Sigma}$ is an equilibrium of $F$, then either
\begin{itemize}
\item $v=\mu_m$, for some $3\le m\le N$,
\item or $v\in\Sigma_{k,m}$, for some $1\le k<m\le N$, and $m\ge 4$,
\item or $v\in\Sigma_{2,3,m}\cup \cup_{k=1}^{m-2}\Sigma_{1,k+1,m}$, for some $1\le k<m\le N$, and $m\ge 4$. 
\end{itemize}
\end{proposition}

This proposition is a consequence of Lemmas \ref{l(H)le3} and \ref{lem:Ekn}.

\blem \label{l(H)le3}
 Let $v\in\Lambda$. Then $|\{i\,:\, v_i>0\}|\in\{1,2,3\}$.
Moreover $|\supp(v)|\ge 3$, and if $|\supp(v)|=3$, then $v$ is uniform on $\supp(v)$.
\elem

\begin{proof} Let $v\in\Lambda$. Then for all $i\in\supp(v)$, $v_i^{\alpha-1}H_i(v)=H(v)$.

Set $c_1=\sum_{j} v_j^\alpha$ and $c_2=\sum_{j\ne k} v_j^\alpha v_k^\alpha$. Then, for all $i$,
$$H_i(v)=c_2-2v_i^\alpha (c_1-v_i^\alpha).$$
Thus, for all $i$, $v_i^{\alpha-1} H_i(v)=f(v_i^\alpha)$ where
$$f(x)=x^\beta \left[2x^2-2c_1x+c_2\right].$$
We have
$$f'(x) = t^{\beta-1}[2(2+\beta) x^2 -2(1+\beta) c_1 x +\beta c_2].$$
Set $\Delta=(1+\beta)^2c^2_1-2\beta(2+\beta)c_2$. Then, when $\Delta\le 0$, $f$ is increasing. And, when $\Delta>0$, setting $x_\pm=\frac{(1+\beta)c_1\pm \sqrt{\Delta}}{2(2+\beta)}$, $f$ is increasing on $[0,x_-]$, decreasing on $[x_-,x_+]$ and increasing on $[x_+,\infty)$. 

For $H>0$, set $\ell(H):=|\{x\ge 0:\, f(x)=H\}|$.
Then, when $\Delta \le 0$, $\ell(H)=1$.
When $\Delta>0$, $\ell(H)=1$ if $H\not\in [f(x_+),f(x_-)]$, $\ell(H)=2$ if $H\in\{f(x_+),f(x_-)\}$ and $\ell(H)=3$ if $f(x_+)<H<f(x_-)$.

Now since for all $i$ such that $v_i>0$, we have $f(v_i^\alpha)=H(v)$, this proves that $|\{i\,:\, v_i>0\}|\in\{1,2,3\}$.
Since $v\in\Sigma$, then $|\supp(v)|\ge 3$. It is straightforward to check that if ${|\supp(v)|=3}$, then $v$ is uniform on $\supp(v)$.
\end{proof}

\blem \label{lem:Ekn} Let $v\in\vec{\Sigma}$ be an equilibrium such that $|\{i:\,v_i>0\}|=3$  and $|Supp(v)|=m\ge 4$. Then $v\in \Sigma_{2,3,m}\cup \cup_{k=1}^{m-2}\Sigma_{1,k+1,m}$. \elem
\begin{proof} Suppose there is an equilibrium $v\in \Sigma_{k,\ell,m}$.
Using the notation of the proof of Lemma \ref{l(H)le3}, we get that $\Delta>0$ and $f(x_+)<H(v)<f(x_-)$. 
Thus, there is $x_1>x_2>x_3$ such that $f(x_1)=f(x_2)=f(x_3)=H(v)$. Note that $x_1>x_+$ and $x_2>x_-$ so that $x_1+x_2>\frac{1+\beta}{2+\beta} c_1$. 
For $j\in\{1,2,3\}$, set $k_j=|\{i:\; v^\alpha_{i}=x_j\}|$. Then $k=k_1$, $\ell=k_1+k_2$ and $m=k_1+k_2+k_3$. 
We also have that $c_1>k_1 x_1+ k_2 x_2$. Therefore, $(2+\beta)(x_1+x_2)>(1+\beta)(k_1 x_1+ k_2x_2)$. 
If $k_1\ge 2$ and $k_2\ge 2$, this implies that  $(2+\beta)(x_1+x_2)>2(1+\beta)(x_1+x_2)$. This is a contradiction. Suppose that $k_1\ge 3$ and $k_2=1$, then we get $(2+\beta)(x_1+x_2)>(1+\beta)(3 x_1+ x_2)>2(1+\beta)(x_1+x_2)$, which is again a contradiction.
This proves the lemma.
\end{proof}
	
\subsection{Stability of the equilibriums of $F$ when $\alpha>1$}
In the following it will be convenient to set $\beta:=\frac{\alpha-1}{\alpha}\in (0,1)$. Our objective in the following sections is to prove the following proposition.
\begin{proposition}\label{prop:Lambdafinite}
When $\alpha>1$, then $\Lambda$ is a finite set. Moreover, if $v\in\Lambda$ and if $\beta_m:=\frac{2}{m-1}$, where $m:=|\supp(v)|$, then
\begin{enumerate}
\item[(i)] $v$ is stable if and only if $v$ is uniform and $\beta<\beta_m$;
\item[(ii)] $v$ is unstable if and only if $v$ is not uniform or if $\beta>\beta_m$.
\end{enumerate}
\end{proposition}
Note that $\beta < \beta_m$ if and only if $m<\frac{3\alpha-1}{\alpha-1}$.
This proposition is a consequence of Lemmas \ref{l(H)le3}, \ref{lem:Ekn}, Propositions \ref{prop:frrw:st_unif_eq,a>1}, \ref{propinst}, \ref{prop:unstable} and \ref{propEkn<infty}.

The following lemma provides useful properties in order to study the stability of an equilibrium. For $v\in\Sigma$ and ${i\in\Vset}$, we will use the convention $H_{i,i}(v)=0$.

\begin{lemma}
\label{lem:frrw:Dx-Dy,a>1}
Let $v$ be an equilibrium. Then $DF(v)$ is a self-adjoint operator on $\mathbb{R}^N$ with respect to the inner product $\langle \cdot,\cdot\rangle_v$ defined by
$\langle x,y\rangle_{v}=\sum_{i\in\Vset:\; v_i>0} x_iy_i/v_i+\sum_{i\in\Vset:\; v_i=0}x_iy_i$.
In particular, setting $I(v):=\{i:\,v_i=0\}$,
\begin{enumerate}
\item[(i)]$DF(v)$ is diagonalisable and its eigenvalues are all real.
\item[(ii)] $-1$ is an eigenvalue with eigenspace containing $v$ and $\{e_i:\; i\in I(v)\}$. 
\item[(iii)] The vector space $E:=\{u\in\mathbb{R}^N:\; \sum_k u_k=0 \hbox{ and $u_i=0$ for all $i\in I(v)$}\}$ is stable for $DF(v)$ and orthogonal to the vector space spanned by $v$ and $\{e_i:\; v_i=0\}$ for the inner product $\langle \cdot,\cdot\rangle_v$.
\item[(iv)] The trace of $DF(v)$ restricted to $E$ is equal to $m(\alpha-1)-(3\alpha-1)$, where $m=|\{i:\,v_i>0\}|$.
\end{enumerate}
Moreover, for all $i,j\in\supp(v)$,
\begin{equation}
\label{eq:frrw:Dx-Dy,a>1}
D_{e_i-e_j}F(v) = \big(\alpha-1\big)(e_i-e_j)+\frac{2\alpha}{H(v)}\sum_{k\in\Vset} v_k^\alpha \left(v_i^{\alpha-1}H_{i,k}(v)-v_j^{\alpha-1}H_{j,k}(v)\right)  e_k.
\end{equation}
\end{lemma}

\begin{proof}
Let $v$ be an equilibrium. 
Recall that (see \eqref{def:frrw:hat pi}, \eqref{eq:frrw:dHj} and \eqref{eq:frrw:dH}) for $i, j\in \Vset$,
\begin{eqnarray*}
F_i(v)&=&-v_i+\frac{v_i^\alpha H_i(v)}{H(v)}\;,\\
\partial_i H_j(v)&=&
\left\{\begin{array}{ll} 0\;, & \hbox{when $i=j$}\;,\\
2\alpha v_i^{\alpha-1}H_{i,j}(v)\;, &\hbox{when $i\ne j$}\;,
\end{array}\right.\\
\partial_i H(v)&=& 3\alpha v_i^{\alpha-1}H_i(v).
\end{eqnarray*}
We also have
\begin{eqnarray*}
\partial_i F_j(v) &=&\left\{\begin{array}{ll}-1+\alpha \frac{v_i^{\alpha-1}H_i(v)}{H(v)}-3\alpha v_i \left(\frac{v_i^{\alpha-1}H_i(v)}{H(v)}\right)^2\;,& \hbox{when $i=j$}\;,\\
v_j\left(2\alpha \frac{v_i^{\alpha-1}v_j^{\alpha-1}H_{i,j}(v)}{H(v)}-3\alpha \frac{v_i^{\alpha-1}H_i(v)}{H(v)}\frac{v_j^{\alpha-1}H_j(v)}{H(v)}\right)\;, & \hbox{when $i\ne j$}.
\end{array}\right.
\end{eqnarray*}
Note first that if $v_i=0$ or if $v_j=0$, then $\partial_iF_j(v)=-\delta_{i,j}$.
Since $v$ is an equilibrium, when $v_i>0$, we have $v_i^{\alpha-1} H_i(v)=H(v)$ and thus
\begin{eqnarray*}
\partial_i F_j(v) &=&\left\{\begin{array}{ll}
\alpha-1 -3\alpha v_i \;,& \hbox{if $j=i$}\;,\\
v_j\left(2\alpha v_i^{\alpha-1}v_j^{\alpha-1}\frac{H_{i,j}(v)}{H(v)}-3\alpha\right)\;, & \hbox{if $j\ne i$ and $v_j>0$\;,}
\end{array}\right.
\end{eqnarray*}
which implies that for $v_i>0$,
\begin{eqnarray*}
\partial_i F(v)
&=& (\alpha-1 -3\alpha v_i)e_i+\sum_{j\neq i} v_j\left(2\alpha v_i^{\alpha-1}v_j^{\alpha-1}\frac{H_{i,j}(v)}{H(v)}-3\alpha\right) e_j\;,\\
&=& (\alpha-1)e_i -3\alpha v+\frac{2\alpha}{H(v)} \sum_{j} v_i^{\alpha-1}v_j^{\alpha}H_{i,j}(v) e_j
\end{eqnarray*}
and thus \eqref{eq:frrw:Dx-Dy,a>1} follows.

Set $a_{i,j}=\frac{\partial_iF_j(v)}{v_j}$ if $v_j>0$ and $a_{i,j}=-\delta_{i,j}$ if $v_j=0$. Then $A=(a_{i,j})$ is a symmetric matrix. Let $\langle \cdot,\cdot\rangle_v$ be the inner product on $\mathbb{R}^N$ defined by
$\langle x,y\rangle_{v}=\sum_{i\in\Vset:\; v_i>0} x_iy_i/v_i+\sum_{i\in\Vset:\; v_i=0}x_iy_i$. Then 
$\langle D_xF(v),y\rangle_{v}=\sum_{i,j} a_{i,j}x_iy_j$, and $DF(v)$ is self-adjoint with respect to this inner product. It is thus diagonalisable with real eigenvalues.

When $v_i=0$, $D_{e_i}F(v)=-e_i$. We also have that
\begin{eqnarray*}
D_vF(v)
&=& (\alpha-1)v -3\alpha v +\frac{2\alpha}{H(v)}\sum_k \sum_j v_j^{\alpha}v_k^{\alpha}H_{j,k}(v)e_j\;,\\
&=&  -v-2\alpha v + 2\alpha\sum_j    \frac{v_j^\alpha H_{j}(v)}{H(v)}e_j \,=\, -v\;.
\end{eqnarray*}
This proves (ii). The proof of (iii) is straightforward. To prove (iv), we have that for $v$ an equilibrium, 
$\hbox{Trace}(DF(v))=\sum_i\partial_i F_i(v)=m(\alpha-1)-3\alpha - (N-m)$ and 
we get (iv) since the trace of $DF(v)$ restricted to $E$ is $\hbox{Trace}(DF(v))+(N-m+1)$ 
(using that $N-m+1$ is the dimension of the vector space spanned by $v$ and $\{e_i:\;v_i=0\}$). \end{proof}

\begin{proposition}\label{prop:frrw:st_unif_eq,a>1}
When $\alpha>1$, a uniform probability measure on a subset of $\Vset$ containing $m\ge 3$ vertices  is stable if and only if $m<\frac{3\alpha-1}{\alpha-1}$ and is unstable if and only if $m>\frac{3\alpha-1}{\alpha-1}$.
\end{proposition}

\begin{proof}
Let $v$ be a uniform measure on a subset of $\Vset$ containing $m\ge 3$ vertices. Suppose that ${\alpha>1}$. We have
\begin{eqnarray*}
H_{i,j}(v)&=& (m-2)m^{-\alpha}\;, \qquad\hbox{ for $i,j\in\supp(v)$ with $i\neq j$}\;, \\
H(v)&=&m(m-1)(m-2)m^{-3\alpha}.
\end{eqnarray*}
Using \eqref{eq:frrw:Dx-Dy,a>1}, for $i,j\in\supp(v)$,
$$D_{e_i-e_j}F(v) = \left(-1+\alpha \left(\frac{m-3}{m-1}\right)\right)(e_i-e_j)$$
and, using Lemma~\ref{lem:frrw:Dx-Dy,a>1}, this completes the description of the spectrum of ${DF(v)}$ : $-1$ is an eigenvalue of multiplicity $N-m+1$ and $-1+\alpha \left(\frac{m-3}{m-1}\right)$ is an eigenvalue of multiplicity $m-1$. 
Hence the proposition.
\end{proof}

\begin{remark}
When $\alpha>1$, uniform probability measures on subsets of $\Vset$ containing exactly three vertices are always stable equilibriums.
\end{remark}

\begin{proposition} \label{propinst}
If $v$ is an equilibrium such that $m:=|\{i:\,v_i>0\}|>\frac{3\alpha-1}{\alpha-1}$, then $v$ is unstable.
\end{proposition}
\begin{proof}
It is a simple consequence of (iv) of Lemma \ref{lem:frrw:Dx-Dy,a>1}. If $m>\frac{3\alpha-1}{\alpha-1}$, then the trace of $DF(v)$ restricted to $E$ is positive ($E$ defined in (iii) of Lemma~\ref{lem:frrw:Dx-Dy,a>1}). Thus there is a positive eigenvalue and $v$ is unstable.
\end{proof}

\subsubsection{Equilibriums in $\Sigma_{k,m}$ and their stability}\label{SIgmakn}
In this section, $k$, $m$ and $\alpha$ are given such that $1\le k<m\le N$ and $m\ge 4$. We set $\ell=m-k$, $\beta=\frac{\alpha-1}{\alpha}$ and $\beta_m=\frac{2}{m-1}$.
\begin{proposition}\label{prop:unstable} There is a finite number of equilibriums in $\Sigma_{k,m}$, and all these equilibriums are unstable. More precisely, 
\begin{enumerate}[(i)]
\item If $k\in\{1,2\}$, then
\begin{itemize}
\item if $\beta\le \beta_m$, there is no equilibrium in $\Sigma_{k,m}$,
\item if $\beta>\beta_m$, there is exactly one equilibrium in $\Sigma_{k,m}$.
\end{itemize}
\item If $3\le k<m/2$, then there is $\beta_{k,m}\in (\beta_m,1)$ such that
\begin{itemize}
\item if $\beta \le \beta_m$ or if $\beta=\beta_{k,m}$, there is exactly one equilibrium in $\Sigma_{k,m}$,
\item if $\beta_m<\beta<\beta_{k,m}$, there are exactly two equilibriums in $\Sigma_{k,m}$,
\item if $\beta > \beta_{k,m}$, there is no equilibrium in $\Sigma_{k,m}$.
\end{itemize}
\item If $m/2\le k\le m$, then
\begin{itemize}
\item if $\beta < \beta_m$, there is exactly one equilibrium in $\Sigma_{k,m}$,
\item if $\beta\ge\beta_m$, there is no equilibrium in $\Sigma_{k,m}$.
\end{itemize}
\end{enumerate}
\end{proposition}
\begin{proof}
For $v\in\Sigma_{k,m}$, set $c=v_m$ and $a=v_1/v_m$. Note that $a>1$, $ac\le\frac13$ and $c=(ka+\ell)^{-1}$.
Then $F_i(v)=F_1(v)$ for ${1\le i\le k}$, ${F_i(v)=F_m(v)}$ for $k+1\le i\le m$ and $F_i(v)=0$ for ${i\ge m+1}$, and
\begin{eqnarray*}
F_1(v)&=&-ac+a^\alpha K_1 K^{-1}\;,\\
F_m(v)&=&-c+K_2 K^{-1}\;,
\end{eqnarray*}
where
\begin{eqnarray}
K_1&=&c^{-2\alpha}H_1(v)\,=\,(k-1)(k-2)a^{2\alpha}+2\ell (k-1) a^\alpha + \ell(\ell-1)\;,\nonumber\\
K_2&=&c^{-2\alpha}H_2(v)\,=\,k(k-1)a^{2\alpha}+2k(\ell-1) a^\alpha + (\ell-1)(\ell-2)\;,\nonumber\\
\label{equat:K}K &=& c^{-2\alpha}H(v)\,=\,ka^\alpha K_1 + \ell K_2.
\end{eqnarray}
Set $u:=\frac{m}{\ell}(\mu_k-\mu_m)$. Then $u_i=k^{-1}$ for ${1\le i\le k}$, $u_i=-\ell^{-1}$ for $k+1\le i\le m$ and $u_i=0$ for ${i\ge n+1}$. Then, using $c=(ka+\ell)^{-1}$ and \eqref{equat:K},
\begin{equation}
F(v)=\frac{k\ell ac}{K}[a^{\alpha-1}K_1-K_2] u.
\end{equation}
Set $x=a^\alpha-1$. Then,
$K_1= (m-1)(m-2) B(x)$ and $K_2= (m-1)(m-2)A(x)$ with
\begin{eqnarray*}
A(x) &=& 1+2a_1x+a_1a_2x^2 \;,\\
B(x) &=& 1+2b_1x+b_1b_2x^2 \;,
\end{eqnarray*}
where
$$\begin{array}{ll}
a_1=\frac{k}{m-1}\;, & a_2=\frac{k-1}{m-2}\;,\\
b_1=\frac{k-1}{m-1}\;, & b_2=\frac{k-2}{m-2}\;.
\end{array}$$
Let $\phi$ be the function defined by $\phi(x)=\frac{\log(A(x))-\log(B(x))}{\log(1+x)}$.
Then
\begin{equation}
F(v)=g(a)[(1+x)^{\beta-\phi(x)}-1] u\;,
\end{equation}
with $g(a)=\frac{k\ell ac K_2}{K}$. Note that $g$ is $C^1$ and positive on $(0,\infty)$.

\medskip
The function $\phi$ is studied in the following lemma.
\begin{lemma}\label{lem:phi}
We have $\lim_{x\to 0+}\phi(x)=\beta_m$, and 
\begin{itemize}
\item If $k\in\{1,2\}$, then $\phi$ is increasing, with $\phi'(x)>0$ for all $x\in (0,\infty)$, and $\lim_{x\to\infty}\phi(x)=\left\{\begin{array}{ll} \infty &\hbox{ if } k=1\;,\\ 1 &\hbox{ if } k=2\;.\end{array}\right.$
\item If $k\ge 3$, then $\lim_{x\to\infty}\phi(x)=0$ and there is $x_1\ge 0$ such that $\phi'(x)>0$ for all $x\in (0,x_1)$ and $\phi'(x)<0$ for all $x>x_1$. Moreover, if $k<m/2$, then $x_1>0$ and if $k\ge m/2$, then $x_1=0$.
\end{itemize}
\end{lemma}
The proof of this lemma is given in the appendix.

\medskip
Since $v$ is an equilibrium if and only if $\phi(x)=\beta$, Lemma~\ref{lem:phi} easily implies (i), (ii) and (iii), taking in (iii), $\beta_{k,m}=\max\{\phi(x):\,x>0\}$.
Since Proposition \ref{propinst} states that equilibriums are unstable when $\beta>\beta_m$,
it thus remains to prove that for $3\le k\le m$ and $\beta \le \beta_m$, the equilibrium in $\Sigma_{k,m}$ is unstable.

So, we suppose $3\le k\le m$ and $\beta \le \beta_m$. In this case, $\phi'(x)<0$ for all $x>x_1$ and $\phi(x)>\beta_m$ for all $x\in (0, x_1]$, thus there is a unique $x_*>0$ such that $\phi(x_*)=\beta$. Set $a_*=(x_*+1)^{\frac{1}{\alpha}}$, then $v_*:=v(a_*)$ is the equilibrium in $\Sigma_{k,m}$.
Note that $\frac{dv(a)}{da}=k\ell c^2 u$ and thus that
$D_uF(v(a))=\frac{1}{k\ell c^2}\frac{d}{da} F(v(a))$.
We have,
$$\frac{d}{da} F(v(a_*)) = -\alpha a_*^{\alpha-1}g(a_*)\phi'(x_*)\log(1+x_*) u.$$
Therefore, there is $\lambda_*>0$ such that $D_uF(v_*)=\lambda_* u$, and $v_*$ is an unstable equilibrium.
\end{proof}

\subsubsection{Equilibriums in $\Sigma_{2,3,m}\cup \cup_{k=1}^{m-2}\Sigma_{1,k+1,m}$ and their stability}
\label{cupEkn}
In this section, $k$, $m$ and $\alpha$ are given such that $1\le k\le m-2\le N$ and $m\ge 4$. We set $\ell=m-k-1$, $\beta=\frac{\alpha-1}{\alpha}$ and $\beta_m=\frac{2}{m-1}$.
When we say that $v\in\Sigma$ belongs, up to a permutation of the indices, to a subset $A$ of $\Sigma$, we mean that there is a permutation $\sigma$ such that $v\circ\sigma\in A$.

Define the mapping $(a,b)\mapsto v(a,b)\in\Sigma$ by $v_1(a,b)=ac$, $v_i(a,b)=bc$ for $2\le i\le k+1$ and $v_i(a,b)=c$ for $k+2\le i\le m$ and $v_i(a,b)=0$ for $m+1\le i\le N$, and where $c=(a+kb+\ell)^{-1}$. 
Note that, if $c<bc<ac\le \frac13$, then $v(a,b)\in \Sigma_{1,k+1,m}$ and if $c<ac<bc\le \frac13$, then (up to a permutation of the indices) $v(a,b)\in \Sigma_{k,k+1,m}$. 
Denote by $E_{k,m}$ the set of all $v\in \Sigma$ such that $v=v(a,b)$ for some $a>1$ and $b>1$ with $a\ne b$ (note that for $a>1$, $v(a,a)\in \Sigma_{k+1,m}$ and $v(a,1)\in\Sigma_{1,m}$ and that for $b>1$, up to a permutation of indices, $v(1,b)\in\Sigma_{k,m}$).

For $v=v(a,b)\in E_{k,m}$, we have $F_i(v)=F_2(v)$ for $2\le i\le k+1$, $F_i(v)=F_m(v)$ for $k+2\le i\le m$ and $F_i(v)=0$ for $m+1\le i\le N$, and
\begin{eqnarray*}
F_1(v)&=&-ac+a^\alpha K_1 K^{-1}\;,\\
F_2(v)&=&-bc+b^\alpha K_2 K^{-1}\;,\\
F_m(v)&=&-c+K_3 K^{-1}\;,
\end{eqnarray*}
where $K_1 = c^{-2\alpha}H_1(v)$, $K_2 = c^{-2\alpha}H_2(v)$, $K_3 = c^{-2\alpha}H_m(v)$ and $K =c^{-2\alpha} H(v)$. We have
\begin{eqnarray*}
K_1&=&k(k-1)b^{2\alpha}+2k\ell b^\alpha + \ell(\ell-1)\;,\\
K_2&=&(k-1)(k-2)b^{2\alpha}+2(k-1)\ell b^\alpha + \ell(\ell-1) + 2a^\alpha[(k-1)b^\alpha+\ell]\;,\\
K_3&=&k(k-1)b^{2\alpha}+2k(\ell-1) b^\alpha + (\ell-1)(\ell-2) + 2a^\alpha[kb^\alpha+(\ell-1)]\;,\\
K &=& a^\alpha K_1 + kb^\alpha K_2 + \ell K_3.
\end{eqnarray*}
Then, $v(a,b)\in E_{k,m}\cup\Sigma_{k+1,m}$ is an equilibrium if and only if $a^{\alpha-1}K_1=b^{\alpha-1}K_2=K_3$.

Set $x=a^\alpha$ and $y=b^\alpha$. Then,

\begin{eqnarray}
\left\{\label{def:K123} \begin{array}{ll}
 K_1 &= P_1(y)\;,\\
K_2 &= P_2(y)+2x[(k-1)y+\ell]\;,\\
 K_3 &= P_3(y)+2x[ky+(\ell-1)],
\end{array}\right.
\end{eqnarray}
where
\begin{eqnarray*}
P_1(y)&=&k(k-1)y^2+2k\ell y + \ell(\ell-1)\;,\\
P_2(y)&=&(k-1)(k-2)y^2+2(k-1)\ell y + \ell(\ell-1)\;,\\
P_3(y)&=&k(k-1)y^2+2k(\ell-1) y + (\ell-1)(\ell-2).
\end{eqnarray*}
Note that 
\begin{eqnarray}
\label{eq:K3-K1} K_3-K_1&=& 2(x-1)[ky+(\ell-1)]\;,\\
\label{eq:K3-K2} K_3-K_2&=&2(y-1)[x+(k-1)y+(\ell-1)]\;, \\
\label{eq:K2-K1} K_2-K_1&=&2(x-y)[(k-1)y+\ell] 
\end{eqnarray}
and that $v:=v(a,b)$ is an equilibrium if and only if $x^\beta K_1=y^\beta K_2=K_3$.

\begin{lemma}
The mapping $f_\beta: (1,\infty)\to [0,\beta]$ defined by  $f_\beta(x)=\frac{x^\beta-1}{x-1}$ is strictly decreasing with $\lim_{x\to 1}f_\beta(x)=\beta$ and $\lim_{x\to\infty} f_\beta(x)=0$.
\end{lemma}
\begin{proof}
For $x>1$, $f'_\beta(x)=\frac{u(x)}{(x-1)^2}$ with $u(x)=-(1-\beta)x^\beta -\beta x^{\beta-1}+1$. 
We have $u'(x)=-\beta(1-\beta)x^{\beta-2}(x-1)<0$, and therefore $u(x)<u(1)=0$. 
This proves that $f_\beta$ is strictly decreasing. 
\end{proof}

\begin{proposition}\ \label{propEkn<infty} \begin{enumerate}[(i)]
\item \label{propEkn<infty:beta>betam} If $\beta>\beta_m$, then there is a finite number of equilibriums in $E_{k,m}$.
\item \label{propEkn<infty:beta<betam} If $\beta\le \beta_m$, then there is no equilibrium in $E_{k,m}.$
\end{enumerate}
\end{proposition}
\begin{proof}[Proof of Proposition \ref{propEkn<infty}-\eqref{propEkn<infty:beta>betam}]
Fix $k\ge 1$ and suppose $\beta>\beta_m$. 
Note that $v(a,b)\in E_{k,m}\cup \Sigma_{k+1,m}$ is an equilibrium if and only if  $x=f_\beta^{-1}\circ h_1(y)=h_2(y)$, where $h_1(y)=\frac{2[ky+(\ell-1)]}{P_1(y)}$ and 
\begin{eqnarray*}
h_2(y)&=&\frac{P_3(y)-y^\beta P_2(y)}{2g(y)}\;,
\end{eqnarray*}
where
$g(y)=(k-1)y^{\beta+1}+\ell y^\beta - ky - (\ell-1)$.
Note that $h_1:[0,\infty)\to\mathbb{R}$ is decreasing, with $h_1(1)=\beta_m$ and $\lim_{y\to\infty}h_1(y)=0$, and $h_2(y)$ is defined for all $y$ such that $g(y)\ne 0$.

Suppose that $v(a,b)$ is an equilibrium not isolated in the set of equilibriums in $E_{k,m}\cup \Sigma_{k+1,m}$.
Then, there is a sequence $(a_i,b_i)$ converging to $(a,b)$, such that $v(a_i,b_i)$ is an equilibrium.
Set $t_i=b_i^\alpha-y$ (with $y=b^\alpha$), then  $t_i$ converges to $0$ and 
$h_1(y+t_i)=f_\beta\circ h_2(y+t_i)$ for all $i$.

Note that necessarily, $h_1(y+t_i)<\beta$ and $h_2(y+t_i)>1$ for all $i$.
The functions $t\mapsto h_1(y+t)$ and $t\mapsto f_\beta\circ h_2(y+t)$ are both expandable in power series, with radius of convergence respectively $r_1$ and $r_2$. Note that $r_1=y$ and that 
$$r_2=\inf\{|t|:\; h_2(y+t)=1 \hbox{ or } g(y+t)=0\hbox{ or } |t|=y)\}.$$
Since $g(1)=0$, we have $r_2\le y-1$.
Since $h_1(y+t_i)=f_\beta\circ h_2(y+t_i)$ for all $i$, the two power series coincide, and we have that $h_1(y+t)=f_\beta\circ h_2(y+t)$ for all $t\in (-r_2,r_2)$. 
This implies that there is $y'\in\{y-r_2,y+r_2\}$ such that $g(y')=0$ and $h_1(y')=0$ or such that $h_2(y')=0$ and $h_1(y')=\beta$. Then since $y'\ge 1$, $0<h_1(y')< \beta_m$, which leads to a contradiction. 
This shows that every equilibrium in $E_{k,m}\cup \Sigma_{k+1,m}$ is isolated.

It can be checked that $\limsup_{y\to 1+} \frac{f_\beta^{-1}\circ h_1(y)}{h_2(y)}<0$ and $\limsup_{y\to \infty} \frac{f_\beta^{-1}\circ h_1(y)}{h_2(y)}<0$. 
Thus the set of equilibrium in  $E_{k,m}\cup \Sigma_{k+1,m}$ is included in a compact subset of  $E_{k,m}\cup \Sigma_{k+1,m}$. This permits to conclude that there is only a finite number of equilibriums in  $E_{k,m}\cup \Sigma_{k+1,m}$.
\end{proof}

\begin{proof}[Proof of Proposition \ref{propEkn<infty}-\eqref{propEkn<infty:beta<betam}]

\textbf{Case $\bm{k=1$}:} In this case, $m=\ell+2$ and $\ell\ge 2$.
Suppose that $v(a,b)\in E_{1,m}$ is an equilibrium. 
Then, using \eqref{def:K123}, we have $K_1=\ell [2y+\ell-1]$ and $K_2=\ell [2x+\ell-1]$.

Now, using \eqref{eq:K3-K1}, $x^\beta K_1=K_3$ is equivalent to $f_{\beta}(x)=f(y):=\frac{2[y+\ell-1]}{\ell [2y+\ell-1]}$.
Then $f$ is decreasing,
with $f(1)=\frac{2}{\ell+1}=\beta_m$ and $\lim_{y\to\infty}f(y)=\frac{1}{\ell}$.
Therefore, if $\beta\le \frac{1}{\ell}$, there is no equilibrium in $E_{1,m}.$

If $\ell^{-1}\le \beta\le \beta_m$, let $x_0>1$ be such that $f_\beta(x_0)=\ell^{-1}$ and define $\psi:[1,x_0)\to [1,\infty)$ by $\psi=f^{-1}\circ f_\beta$.
Note that $\psi$ is increasing, with $\psi(1)=f^{-1}(\beta)>1$ and $\lim_{x\to x_0}\psi(x)=\infty$.
We have that $x^\beta K_1=K_3$ is equivalent to $y=\psi(x)$ and $1<x<x_0$. 
We also have that $y^\beta K_2=K_3$ is equivalent to $x=\psi(y)$ and $1<y<x_0$. 

Note that $v(a,a)$ is an equilibrium in $\Sigma_{2,m}$ if and only if $\psi(x)=x$ and $1<x:=a^\alpha<x_0$.  
Proposition~\ref{prop:unstable} shows that when $\beta\le\beta_m$, there is no equilibrium in $\Sigma_{2,m}$. 
Thus, $\psi(x)\ne x$ for all $x\in (1,x_0)$ and, since $\psi(1)>1$, we have $\psi(x)>x$ for all $x\in (1,x_0)$.
Now, if $v(a,b)\in E_{1,m}$ is an equilibrium, then $y=\psi(x)>x=\psi(y)>y$, which is a contradiction.

\textbf{Case $\bm{k=2}$:} Note that $m=\ell+3$ and $\ell\ge 1$.
Suppose that $v(a,b)\in E_{2,m}$ is an equilibrium. 
Then, since $k=2$, we have $K_1=G(y,y)$ and $K_2=G(x,y)$, where
$$G(x,y)=2xy+2\ell(x+y)+\ell(\ell-1).$$ 
Note also that, using \eqref{eq:K3-K1} and \eqref{eq:K3-K2}, we have $K_3-K_1=2(x-1)[2y+\ell-1]$ and $K_3-K_2=2(y-1)[x+y+\ell-1]$.
Therefore, $x^\beta K_1=K_3$ is equivalent to $f_\beta(x)=\Phi(y,y)$ and $y^\beta K_2=K_3$ is equivalent to $f_\beta(y)=\Phi(x,y)$, where
$$\Phi(x,y)=\frac{2[x+y+\ell-1]}{2xy+2\ell(x+y)+\ell(\ell-1)}.$$
Note that $\Phi(x,y)=\Phi(y,x)$ and that $x\mapsto \Phi(x,y)$ is decreasing, with $\Phi(1,y)=\frac{2(y+\ell)}{(\ell+1)(2y+\ell)}$ and $\lim_{x\to\infty}\Phi(x,y)=\frac{1}{y+\ell}$. This implies that $y\mapsto \Phi(y,y)$ is also decreasing,
with $\Phi(1,1)=\frac{2}{\ell+2}=\beta_m$ and $\lim_{y\to\infty}\Phi(y,y)=0$.

Note that if for $x>1$, we have $\Phi(x,x)=f_\beta(x)$, then $v(a,a)$ (with $a=x^{\frac{1}{\alpha}}$), which (up to a permutation of indices) belongs to $\Sigma_{3,m}$, is an equilibrium.
Proposition~\ref{prop:unstable} states that there is exactly one such equilibrium when $\beta\le\beta_m$. 
So when $\beta\le\beta_m$, there is $x_*>1$ such that  $\Phi(x,x)>f_\beta(x)$ if $x<x_*$ and $\Phi(x,x)>f_\beta(x)$ if $x>x_*$.

Note that if for $y>1$, we have $\Phi(1,y)=f_\beta(y)$, then $v(1,b)$ (with $b=y^{\frac{1}{\alpha}}$), which (up to a permutation of indices) belongs to $\Sigma_{2,m}$, is an equilibrium.
Proposition~\ref{prop:unstable} states that there is no such equilibrium when $\beta\le\beta_m$. 
This implies that for all $y>1$, $\Phi(1,y)>f_\beta(y)$.

A first consequence of these facts is that there is an increasing  continuous function $\psi_1:[1,\infty)\to [1,\infty)$,  
such that if $f_\beta(x)=\Phi(y,y)$ then $y=\psi_1(x)$. Moreover, $\psi_1(1)>1$, $\psi_1(x)\sim 2x^{1-\beta}$ as $x\to\infty$ and
$x<\psi_1(x)<x_*$ if $1<x<x_*$ and $x_*<\psi_1(x)<x$ if $x>x_*$.

A second consequence is that for all $y\ge 1$ such that $(y+\ell)^{-1}<f_\beta(y)\le\Phi(1,y)$, there is a unique $x=\psi_2(y)\ge 1$ such that $f_\beta(y)=\Phi(x,y)$. 
Moreover, $\psi_2(y)>y$ if $1<y<x_*$ and $\psi_2(y)<y$ if $y>x_*$.

Therefore, if $v(a,b)\in E_{2,m}$ is an equilibrium, then  $\psi_1(x)=y$, $\psi_2(y)=x$ and $x\ne y$.
If $1<x<x_*$, then $y=\psi_1(x)\in (x,x_*)$ and thus $x=\psi_2(y)>y$, which is a contradiction.
If $x>x_*$, then $y=\psi_1(x)\in (x_*,x)$ and thus $x=\psi_2(y)<y$, which is again a contradiction.
The lemma is proved.

\textbf{Case $\bm{k\ge3}$:} By Lemma \ref{lem:Ekn}, we have that if $v\in E_{k,m}$ is an equilibrium, then $v\in \Sigma_{1,k+1,m}$.
Recall that $m=k+\ell+1$ and that $\beta_m=\frac{2}{m-1}=\frac{2}{k+\ell}$.
Let $v=v(a,b)\in \Sigma_{1,k+1,m}$. Set again $x=a^\alpha$ and $y=b^\alpha$. Then $x>y>1$.
If $v$ is an equilibrium, then $x^\beta K_1=y^\beta K_2$. Using \eqref{eq:K2-K1}, we thus have
$$\frac{x^\beta-y^\beta}{x-y}=\frac{2y^\beta[(k-1)y+\ell]}{K_1}.$$
When $x>y$, $\frac{x^\beta-y^\beta}{x-y}=f_\beta(x/y) y^{\beta-1} <\beta y^{\beta-1}$ 
and we thus have $f(y):=\frac{2y[(k-1)y+\ell]}{P_1(y)}<\beta$. 
Now, $y\mapsto f(y)$ is increasing on $[1,\infty)$, with $f(1)=\frac{2}{k+\ell}=\beta_m\ge \beta$, which gives a contradiction.
\end{proof}

\subsection{Proof of Theorem \ref{th:frrw:vn CV}}\label{prfthm}
In the previous sections, we have shown that there is a Lyapounov function,  that there is a finite number of equilibriums and characterized the stable equilibriums (see Proposition \ref{prop:Lambdafinite}). 
Therefore, to prove Theorem \ref{th:frrw:vn CV}, we will apply Corollary \ref{cor:lyapunov}, Theorem \ref{th:st_cv} and Theorem \ref{th:unst_non_cv}.

\subsubsection{Convergence towards stable equilibriums}

\begin{proposition}
\label{prop:frrw:equi_att} Uniform probability measures on subsets of $\Vset$ containing $m\ge 3$ vertices are attainable.
\end{proposition}

\begin{proof}
Let $v$ be a uniform measure on $A\subset\Vset$ with $m:=|A|\ge 3$. To prove that $v$ is attainable, we remark that with positive probability, the walk $X$ remains in $A$ and visits the $m$ vertices of $A$ uniformly and always in the same order. Let us write this more precisely.

Without loss of generality, we suppose that $A=\{1,\dots,m\}$. Let $(x_n)_{n\in\Nset}$ be a sequence of vertices, such that for all integer ${n\ge 1}$ and all vertex $i\in\{1,\dots,m\}$, $x_{nm+i}=i$. Denote $\Omega_n=\{\forall q\le nm, X_q=x_q\}$, the event, where during the $nm+1$ first steps $X$ stays on $\{1,\dots,m\}$ and visits the $m$ vertices always in the order $(1,2,\cdots, m,1,\cdots)$.
Note that for all $n\ge1$, $\Pr(\Omega_{n})>0$. Indeed,
\begin{eqnarray*}
\Pr(\Omega_{n})&=&P(v_0)(X_0,1) \prod_{q=0}^{n-1} \prod_{i=1}^{m-1} P(v_{mq+i})(i,i+1) \prod_{q=1}^{n-1}P(v_{mq})(i,1)\;.
\end{eqnarray*}
Since $G$ is a complete graph and $v_n(i)>0$ for all $n\ge1$ and $i\in\{1,\dots,m\}$, we have ${P(v_n)(i,j)>0}$, for all $n\ge1$ and all $i,j\in\{1,\dots,m\}$, such that $i\neq j$.

On the event $\Omega_{n}$, it holds that 
\begin{eqnarray*}
\norm{v_{nm}-v}&=&\max\left( \abs{\frac{n+1}{N+nm}-\frac{1}{m}},\frac{1}{N+nm}\right)
\;\le\; \frac{N}{m^2n}.
\end{eqnarray*}
Thus for all $\epsilon >0$ and $n_0\in\Nset$, there exits $n_1\ge n_0/m$, such that on $\Omega_{n_1}$, ${\norm{v_{n_1m}-v} < \epsilon}$. Therefore,
${\Pr(\exists n\ge n_0,\norm{v_{n}-v}< \epsilon)\ge \Pr(\Omega_{n_1}) >0}.$
\end{proof}

Theorem \ref{th:st_cv} implies the following statements: 
when $\alpha=1$, $v_n$ has a positive probability to converge towards the uniform probability measure on $\Vset$ 
(see Proposition~\ref{prop:frrw:eq,a=1}). When $\alpha>1$, $v_n$ has a positive probability to converge towards a uniform probability measure on a set containing less than $\frac{3\alpha-1}{\alpha-1}$ vertices (see Proposition~\ref{prop:frrw:st_unif_eq,a>1}).

\subsubsection{Localization on the supports of stable equilibrium}

Following~\cite{Benaim2013}, we prove that for $v$ a stable equilibrium, on the event $\{\lim_{n\to\infty} v_n = v\}$,  the walk $X_n$ localizes almost surely on $\supp(v)$, i.e. the set of infinitely often visited vertices by $X_n$ is $\supp(v)$.

\begin{proposition}
Let $v$ be a stable equilibrium, then  on the event ${\{\lim_{n\to\infty} v_n = v\}}$, the set $\Vset\setminus\supp(v)$ is visited almost surely only finitely many times.
\end{proposition}

This proposition is a consequence of the two following lemmas:

\begin{lemma}
\label{lem:localization1}
There exists $\nu>0$ such that, a.s. on the event $\{\lim_{n\to\infty} v_n = v\}$, $$\lim_{n\to\infty} n^\nu \norm{v_n-v} =0.$$
\end{lemma}

\begin{lemma}
\label{lem:localization2}
For any $I\subseteq \Vset$ and $\nu\in(0,1)$, a.s. on the event
$$E_\nu := \{\lim_{n\to\infty} v_n(i)n^\nu =0,\forall i\in I\}\;,$$
the set $I$  is visited only finitely many times.
\end{lemma}

We do not give the proofs of the two previous lemmas here. They can be proved following the lines of the proofs of Lemma~3.13. and Lemma~3.14 of \cite{Benaim2013}.
 
\subsubsection{Non convergence towards an unstable equilibrium}
Let $v_*$ be an unstable equilibrium. Then $v_*\not\in\Sigma_3$.
In this section, it is shown that a.s., $v_n$ does not converge towards $v^*$.
Let $f$ an unstable direction of $v_*$. 
Since $P:\Sigma\setminus\Sigma^3\to \mcm_{ind}$ is $\mcc^1$, then Hypothesis~\ref{hyp:unstb_eq}-\eqref{hyp:unstb_eq:ball} holds for ${v_*}$.

Recall the definitions of $\mca$, $\mca_i$, $\mca_{i,j}$ and $\mcr_*$ given in Section~\ref{sec:non_conv},  that $\pi_1$ and $\pi_2$ denote the marginals of $\pi(v_*)$ and  that, $v_*$ being an equilibrium, $\pi^V(v_*)=v_*$.

\begin{remark}\label{rem:pa*}
The way $P(v_*)$ is defined implies that $\mca_{i,j}=\mca\setminus\{i,j\}$, $\mca_{i}=\mca\setminus\{i\}$ and $\mcr_*=\mca\times\mca$, for all $(i,j)\in\OEset$.
\end{remark}

By using Remark \ref{rem:pa*} and the fact that $|\supp(v_*)|\ge 4$,  Hypothesis~\ref{hyp:unstb_eq}-\eqref{hyp:unstb_eq:A_y}  holds.

\begin{lemma}
Hypothesis~\ref{hyp:unstb_eq}-\eqref{hyp:unstb_eq:G_g} holds.
\end{lemma}

\begin{proof}
Suppose that there exist a constant $C$ and a map $g:\mca\to\Rset$ such that 
\begin{eqnarray*}
Vf(i,j)&=&C+g(i)-g(j),\; \hbox{ for all $(i,j)\in\mcr_*$.}
\end{eqnarray*}
Calculate for $i,j\in\mca$,
$$Vf(i,j)\;=\;\sum_{k} V((i,j),k) f(k)
\;=\;\sum_{k} \delta_j(k) f(k)\;=\;f(j).$$
Thus for all $i,j\in\mca$, 
$f(j)=C+g(i)-g(j)$.
This implies that $g$ is constant on $\mca$ and thus that $f$ is constant. Since $f\in T_0\Delta_\Vset$,  $\sum_{i\in\Vset} f(i)=0$. Therefore $f(i)=0$ for all $i\in\mca$, which is impossible.
\end{proof}

This last lemma achieves the proof of Theorem~\ref{th:frrw:vn CV}. Indeed, Hypotheses~\ref{hyp:unstb_eq} are satisfied and Theorem~\ref{th:unst_non_cv} can be applied.

\section*{Appendix: Proof of Lemma \ref{lem:phi}}

It is easy to check that $\lim_{x\to 0+}\phi(x)=\beta_m$ and that 
$\lim_{x\to\infty}\phi(x)=\left\{\begin{array}{ll} 1 & \hbox{if } k=1\\ +\infty & \hbox{if } k=2\\ 0& \hbox{if }k\ge 3\end{array}\right.$.

\textbf{Case $\bm{k=1}$:}
In this case, $\phi(x)=\frac{\log(1+\beta_m x)}{\log(1+x)}$. We have
\begin{eqnarray*}
\phi'(x)=\frac{\beta_m\psi(x)}{(1+x)(1+\beta_mx)(\log(1+x))^2}\;,
\end{eqnarray*}
where $\psi(x):=(1+x){\log(1+x)}-\frac{1+\beta_mx}{\beta_m}\log(1+\beta_m x)$. We have $\psi'(x)={\log(1+x)}-\log(1+\beta_mx)$ and $\psi(0)=0$.
Since $\beta_m\le 2/3<1$, $\phi$ is increasing.

\textbf{Preliminary for case $\bm{k\ge 2}$:}
Suppose now $k\ge 2$. Set $s=(m-2)^{-1}$ and $t=(k-1)^{-1}$. Then $0<s\le t\le 1$.
Set $\Phi(y)=\phi(y/b_1)$. Then
\begin{eqnarray}
\label{def:phi(y)}
\beta=\Phi(y):=\frac{\log P(y) - \log Q(y)}{\log(1+\lambda y)}\;,
\end{eqnarray} where $\lambda=(s^{-1}+1)t=1/b_1$ and
\begin{eqnarray*}
P(y)&=&1+2(1+t)y+(1+t)(1+s)y^2\;,\\
Q(y)&=&1+2y+(1-t)(1+s)y^2.
\end{eqnarray*}
Set $u(y) = \log P(y) - \log Q(y)$. Since $0\le s\le t\le 1$,  $u(y)>0$ for all $y>0$. 
Moreover
$$\Phi'(y)=\frac{v(y)}{(1+\lambda y)(\log(1+\lambda y))^2}\;,$$
where $v(y)=(1+\lambda y)\log(1+\lambda y) u'(y) - \lambda u(y)$.
Note that
$$v'(y)=\log(1+\lambda y)  [(1+\lambda y)u']'(y).$$
Therefore $v'(y)>0$ if and only if $\frac{d}{dy}\big[(1+\lambda y)u'(y)\big]^{-1} < 0.$
Set
$$R(y)=1+2(1+s)y+(1+s)(1+t)y^2.$$
Then $u'(y)=2t\frac{R}{PQ}(y)$ and thus $v'(y)>0$ if and only if $q(y)<0$, where
$$q(y)=(1+\lambda y)^2\frac{d}{dy} \frac{PQ(y)}{(1+\lambda y)R(y)}.$$

Set $P_0=P-R$ and $Q_0=R-Q$, then 
\begin{eqnarray*}
P_0(y) &=&  2(t-s)y\;,\\
Q_0(y) &=& 2sy+2(1+s)t y^2\,=\,2sy(1+\lambda y).
\end{eqnarray*}
Since $PQ=PR-PQ_0=PR-RQ_0-P_0Q_0$, we get $q(y)=q_1(y)-q_2(y)$, where
\begin{eqnarray*}
q_1(y) &= & (1+\lambda y)^2\frac{d}{dy} \frac{(P-Q_0)(y)}{(1+\lambda y)}\;,\\
q_2(y) &=& (1+\lambda y)^2 \frac{d}{dy} \frac{P_0Q_0(y)}{(1+\lambda y)R(y)}\;.
\end{eqnarray*}

Computing $q_1$ and $q_2$ gives\,:
\begin{eqnarray*}
q_1(y)&=& (2-t/s)(1-s)+(1-t)(1+s)y(2+\lambda y)\;,\\
q_2(y) &=& 8(t-s)s\times\frac{1+(1+s)y}{R^2(y)}\times y (1+\lambda y)^2\;.
\end{eqnarray*}

\textbf{Case $\bm{k=2}$:}
In this case, we have $t=1$ and $s\le 1/2$. Therefore, $q_1(y)=(2-1/s)(1-s)\le 0$. Since $q_2(y)>0$, we get $q(y)<0$.
This shows that ${v'(y)>0}$ for all $y>0$. 
Since $v(0)=0$, $v(y)>0$ for all $y>0$ and therefore $\Phi$ (and hence $\phi$) is increasing on $(0,\infty)$.

\textbf{Case $\bm{k\ge3}$:}
We now suppose $k\ge 3$. 

\blem \label{lem:qconvexe} For all $0<s<t\le 1/2$,  $q$ is a strictly convex function on $(0, \infty)$. \elem
This lemma is proved at the end of this section.

We have that $q'_1(0)=2(1-t)(1+s)$, $q'_2(0)=8(t-s)s$, and thus $q'(0)=2(1-t)(1+s)-8(t-s)s$.  
Since $0<s<t\le 1/2$, $q'(0)>0$ and $q$ is increasing on $\Rset^+$.
Note that $q(0)=(2-t/s)(1-s)$. 

When $k\ge m/2$, then $s\ge t/2$ and $q(0)\ge 0$. We thus have $q(y)>0$ for all $y>0$ and as a consequence  $v'(y)<0$ for all $y>0$. Since $v(0)=0$, this proves that $\Phi'(y)<0$ for all $y>0$. And Lemma \ref{lem:phi} is proved in this case.

When $k< m/2$,  then $s< t/2$ and $q(0)< 0$. Since $q$ is convex and since $\lim_{y\to\infty} q(y)=+\infty$, there is $y_0$ such that $q(y)<0$ if $y<y_0$ and $q(y)>0$ if $y>y_0$. 
Thus $v$ is increasing on $(0,y_0)$ and  decreasing on $(y_0,\infty)$. 
Since $v(0)=0$ and $\lim_{y\to\infty} v(y)<0$, there exists $y_1>y_0$ such that $v(y)>0$ on $(0,y_1)$ and $v(y)<0$ on $(y_1,\infty)$. 
And Lemma \ref{lem:phi} is also proved in this case.

\begin{proof}[Proof of Lemma~\ref{lem:qconvexe}] Firstly, we have for all $y>0$,
$$q''_1(y)=2(1-t)(1+s)^2(t/s)>0.$$

We now upperbound $q''_2(y)$ for all $y>0$.
Set $z=(1+s)y$, $c=t/s$, ${d=(1+t)/(1+s)}$, $D(z)=1+2z+dz^2$ and
$$Q(z)=\frac{z(1+z)(1+cz)^2}{D^2(z)}.$$

Then $q_2(y)=\frac{8(t-s)s}{1+s}\times Q((1+s)y)$ and $q''_2(y)=8(t-s)s(1+s)Q''(z)$. 
Set $L(z)=4(t-s)s^2(1+s)^{-1}Q''(z) D^4(z)$. Then $q''_2(y) < q''_1(y)$ for all $y>0$ as soon as $\frac{q''_2(y)}{q''_1(y)}<1$, then as soon as  $L(z)< t(1-t)D^4(z)$ for all $z\ge 0$. Computing $Q''(z)$, we get that $L$ is a polynomial of degree $5$\,:
$L(z)=\sum_{i=0}^5 \ell_i z^i$, with
\begin{eqnarray*}
\ell_0 &=& 8s(t-s)(2t-3s)/(1+s)\;,\\
\ell_1 &=& 8(t-s)(3t^2-(2-3t)st-(6+8t)s^2)/(1+s)^2\;,\\
\ell_2 &=& 16t(t-s)(3t-(8+5t)s)/(1+s)^2\;,\\
\ell_3 &=& -16(t-s)(4t^3+4st(2+t+t^2)-s^2(3-2t-t^2))/(1+s)^3\;,\\
\ell_4 &=& -8(t-s)(6t^2+10t^3+2st(1-t+2t^2)-s^2(3-2t-t^2))/(1+s)^3\;,\\
\ell_5 &=& -8t(t+1)(t-s)(3t-2s+2st-t^2)/(1+s)^3\;.
\end{eqnarray*}

Using $0<s\le t\le 1/2$, it is easy to check that $\ell_3\le 0$, $\ell_4\le 0$ and $\ell_5\le 0$. Since ${2-3t\ge 0}$, we thus get
\begin{eqnarray*}
L(z)&\le& \frac{8s(t-s)(2t-3s)}{(1+s)} + \frac{24t^2(t-s)}{(1+s)^2}(z+2z^2)\;,\\
&\le& \frac{24t^2(t-s)}{(1+s)^2}\left(\frac{s(2t-3s)(1+s)}{3t^2} +z+2z^2\right)\;.
\end{eqnarray*}
Set $d(t,u):=\frac{1+t}{1+tu}$, $h_1(t,u):=\frac{1-u}{(1+tu)^2}$ and $h_2(t,u):=\frac{u(2-3u)(1+tu)}{3}$. 
Thus, setting also ${u=s/t\in (0,1]}$,
$$L(z)\le 24 t^3 h_1(t,u)\big(h_2(t,u)+z+2z^2\big).$$
In the following, $d(t,u)$, $h_1(t,u)$ and $h_2(t,u)$ will simply be denoted by $d$, $h_1$ and $h_2$.

For all $z > 0$,
$$D^4(z) > 1+8z+(4d+24)z^2.$$
Thus, if for all $z\ge 0$,
\begin{equation}\label{d4z}1+8z+(4d+24)z^2\ge 12h_1(h_2+z+2z^2)\;,\end{equation}
then for all $z > 0$,
$$D^4(z) > 12h_1(h_2+z+2z^2).$$
Let us now prove that \eqref{d4z} holds for all $z\ge 0$. Note that this is equivalent to show that for all $z\ge 0$,
\begin{equation}\label{d4zb}(1-12h_1h_2) + 4(2-3h_1)z + 4(d+6-6h_1)z^2 \ge 0.\end{equation}
Note that $12h_1h_2\le 1$ (since $(1-12h_1h_2)(t,u)\ge (1-12h_1h_2)(0,u)=1-8u+20u^2-12u^3>0$ for all $u\in [0,1]$).
Thus \eqref{d4zb} is satisfied for all $z\ge 0$ as soon as  $2-3h_1\ge 0$ or as soon as $(2-3h_1)^2\le (1-12h_1h_2)(d+6-6h_1).$

Note that  $2-3h_1\ge 0$  when
$$u\ge u_t:=  \frac{2}{\sqrt{3}\sqrt{3+8t+8t^2}+3+4t}.$$

Suppose now that $u<u_t$. 
Then $(2-3h_1)^2\le (1-12h_1h_2)(d+6-6h_1)$ if and only if $F(t,u)\ge 0$, where
\begin{eqnarray*}
G(t,u)&:=&\big( (1-12h_1h_2)(d+6-6h_1) - (2-3h_1)^2\big)\times (1+tu)^4\;,\\
&=&4u -37u^2 +108u^3-72u^4\\
&&+\,t(1+15u-104u^2+220u^3-48u^4-72u^5)\\
&&+\,t^2(3u+5u^2-118u^3+356u^4-228u^5)\\
&&+\,t^3(3u^2+u^3-28u^4+108u^5-72u^6)\\
&&+\,t^4u^3(1+2u).
\end{eqnarray*}

Using the fact that $u< u_t \le 1/3$, 
\begin{eqnarray*}
G(t,u)&\ge& u(4 -37u +108u^2-72u^3)\\
&&+\,t(1+15u-104u^2+196u^3)\\
&&+\,t^2u(3+5u-118u^2+280u^3)\\
&&+\,t^3u^2(3+u-28u^2+84u^3)\\
&&+\,t^4u^3(1+2u).
\end{eqnarray*}
We check that each of the $5$ terms lowerbounding $G(t,u)$ are positive for all ${u\in (0,1/3]}$.
We have thus proved that $(2-3h_1)^2 \le (1-12h_1h_2)(d+6-6h_1)$ for all $u\in [0,u_t]$.
And as a consequence that \eqref{d4z} holds for all $z\ge 0$.

We can now show that $q$ is strictly convex. Inequality \eqref{d4z} implies that $L(z) < {2t^3 D^4(z)}$.
In order to show that $q''_2<q''_1$, it just remains to remark that $2t^3\le t(1-t)$ for all $t\le 1/2$. 
Therefore $q$ is strictly convex.
\end{proof}



\bibliography{bib_LeGoff_Raimond}{}
\bibliographystyle{ieeetr}
\end{document}